\documentclass[12pt]{aptpub}
\usepackage{color}
\renewcommand{\crnotice}{\hfill{{\small\em

\/}\ {\small}}}

\usepackage{amsmath,amstext,url}

\renewcommand{\authornames}{\underline{\footnotesize }
Y.Y. Li
 }
\renewcommand{\shorttitle}
{ \underline{Decay property of a class of $d$-dimensional Markov processes}}

\numberwithin{equation}{section} 

\makeatletter \@addtoreset{equation}{section}

\makeatletter \@addtoreset{lemma}{section}

\makeatletter \@addtoreset{theorem}{section}

\makeatletter \@addtoreset{proposition}{section}

\makeatletter \@addtoreset{corollary}{section}

\makeatletter \@addtoreset{remark}{section}

\makeatletter \@addtoreset{definition}{section}

\makeatletter \@addtoreset{example}{section}

\begin{document}

\thispagestyle{firstpg}

\vspace*{1.5pc} \noindent \normalsize\textbf{\Large {Decay property of a class of $d$-dimensional Markov processes}} \hfill \vspace{0.5cm}

\hspace*{0.75pc}{\small\textrm{\uppercase{Yanyun Li
}}}\hspace{-2pt}$^{*}$ {\small\textit{Central South University}}



\par
\footnote{\hspace*{-0.75pc}$^{*}\,$Postal address:
School of Mathematics and Statistics, Central South University, Changsha, 410083, China. E-mail address:\
li.yanyun@csu.edu.cn}


\par
\renewenvironment{abstract}{%
\vspace{2pt} \hspace*{2.25pc}
\begin{minipage}{14cm}
\footnotesize
{\bf Abstract}\\[1ex]
\hspace*{0.5pc}} {\end{minipage}}
\begin{abstract}
In this paper, we consider the decay property of a special class of $d$-dimensional Markov processes, which can be viewed as a stopped network with the external customer being blocked to empty nodes. The exact value of the decay parameter $\lambda_{\mathcal{C}}$ is obtained by using a new method. It is proved that the process is $\lambda_{\mathcal{C}}$-transient. The corresponding
$\lambda_{\mathcal{C}}$-invariant measures and quasi-distributions are also
presented. Finally, an example on auto quick repair service network is presented to illustrate the results obtained in this paper.
\end{abstract}
\par
\vspace*{12pt} \hspace*{2.0pc}
\parbox[b]{26.75pc}{{
}} {\footnotesize {\bf Keywords:} $d$-dimensional Markov process, decay parameter, invariant measures, \\
\hspace*{8.50pc}subinvariant measures,
subinvariant vectors, quasi-distributions\\

}
\par
\normalsize

\renewcommand{\amsprimary}[1]{
     \vspace*{8pt}
     \hspace*{2.25pc}
     \parbox[b]{24.75pc}{\scriptsize
    AMS 2000 Subject Classification: Primary 60J27 Secondary 60J35
     {\uppercase{#1}}}\par\normalsize}
\renewcommand{\ams}[2]{
     \vspace*{8pt}
     \hspace*{2.25pc}
     \parbox[b]{24.75pc}{\scriptsize
     AMS 2000 SUBJECT CLASSIFICATION:\ \ PRIMARY
     {\uppercase{#1}}\\
     \phantom{
     AMS 2000 SUBJECT CLASSIFICATION:}
 SECONDARY
 {\uppercase{#2}}}\par\normalsize}

\ams{60J27}{60J35}

\par
\vspace{5mm}
 \setcounter{section}{1}
 \setcounter{equation}{0}
 \setcounter{theorem}{0}
 \setcounter{lemma}{0}
 \setcounter{corollary}{0}
\noindent {\large \bf 1. Introduction}
\vspace{3mm}
\par
The decay property plays an important role in the study of continuous-time and discrete-time Markov chains. The existence of decay parameter was firstly revealed by Kingman~\cite{Kin63} which showed that if $C$ is an irreducible class of a continuous-time Markov chain on the countable state space $\mathbf{E}$, then there exists a nonnegative number $\lambda_C$, called the decay parameter of the corresponding process, such that for all $i,j\in C$,
\begin{eqnarray*}
  \lim_{t\to \infty}t^{-1}\log p_{ij}(t)=-\lambda_C,
\end{eqnarray*}
where $P(t)=(p_{ij}(t):i,j\in \mathbf{E})$ is the transition probability function of the corresponding continuous-time Markov chain. It can be proved that this decay parameter can be expressed as
\begin{eqnarray}\label{eq1.1}
  \lambda_C=\inf \{\lambda \geq 0: \int_0^{\infty}e^{\lambda
  t}p_{ij}(t)dt=\infty\}=\sup\{\lambda \geq 0: \int_0^{\infty}e^{\lambda
  t}p_{ij}(t)dt<\infty\},
\end{eqnarray}
where the latter two quantities in (\ref{eq1.1}) are independent of $i,j\in C$. Based on this
pioneer and remarkable work, there are many important research works regarding decay property, such as Flaspohler~\cite{Fla74}, Pollett~\cite{P86,P88},
Darroch and Seneta~\cite{DS67}, Kelly~\cite{KE83}, Kijima~\cite{KJM63}, Nair and Pollett~\cite{NP93}, Tweedie~\cite{T74}, Van
Doorn~\cite{Van85,Van91}.
\par
The deep relationship between invariant measures and
quasi-stationary distributions has been revealed by Van Doorn~\cite{Van91}, and Nair and Pollett~\cite{NP93}. Although the theory of continuous-time Markov chains has flourished, the decay properties of queuing models have not been fully studied. In fact, for most transient Markov processes, even calculating the exact value of the decay parameter remains open yet.
\par
The main purpose of this paper is to consider the decay property of a special class of $d$-dimensional Markov processes, where $d \geq 2$. Its evolution can be described as follows:
\par
(i)\ The system consists of $d$ service nodes, labeled $1,2,\cdots,d$, where $d<\infty$. For each $r=1,2,\cdots,d$, when the $r$th node is nonempty, the entrance of this node is open and the arrival of external customers (i.e., the customers from out side) to this node follows a compound Poisson process with arrival rates $\lambda_{rl}\ (l\geq 1)$; while when the $r$th node is empty, the external customers are not allowed to enter this node. Therefore, $\lambda_r:=\sum\limits_{l=1}^{\infty}\lambda_{rl}$ is the whole arrival rate of customers at the $r$th node. We assume that $\lambda_r<\infty\ (r=1,2,\cdots,d)$;
\par
(ii)\ There is exactly one sever with service rate $\mu_{_r}$ at node $r$. When a customer at node $r$ has been served, he/she leaves the system with probability $\gamma_{_{r0}}$ or transfer to node $s$ with probability $\gamma_{_{rs}}$ ($r,s=1,\cdots,d$), where $\sum\limits_{s=0}^d\gamma_{_{rs}}=1$ for each $r$;
\par
(iii)\ The arrival of customers and service at each node are all independent of each other;
\par
(iv)\ When the network is empty, it stops.
\par
In general situation, the customers from outside may be allowed to enter an empty node. However, in practical applications, there still exist some practical scenarios that a node does not allow external customers when it is idle. For example, consider an automobile quick repair network with $m$ service stations. The entrance of each station is only open when there are customers at that station. However, in order to alleviate the burden on other stations, each station must accept customers transferred from other stations at all times. Another example is hospital department system, different departments can be regarded as different nodes. When a department is idle, the doctor needs to deal with internal affairs (such as medical records, participating in internal meetings, etc.), and then the department does not receive external patients. However, in order to reduce the workload of other departments, this department still accepts patients transferred from other nodes. Such rule can ensure that the doctor can focus on diagnosing the patients continuously. Ayse and Surendra~\cite{AS99} considered an open finite manufacturing/queueing networks with $N$-policy, where a station  is assigned to alternative jobs at its idle time till the accumulated work at the station reaches a predetermined level of $N$ jobs.
\par
By the definitions of decay parameter and quasi-stationary distribution (see Definition~\ref{def1.2} below), we can see that the decay property and the decay parameter and quasi-stationary distribution of a stopped Markov queue describe the degree of busyness during its busy period.
Li and Chen~\cite{LC08} considered the decay property of one-dimensional Markov queueing system with idle restriction.
\par
For convenience of our discussion, we adopt the following conventions throughout this paper:
\par
(C-1)\ $\mathbf{Z}^d_+=\{(i_1,\cdots,i_d): i_1,\cdots,i_d\in \mathbf{Z}_+\}$.
\par
(C-2)\ $[0,1]^d=\{(x_1,\cdots,x_d):0\leq x_1,\cdots,x_d\leq 1\}$.
\par
(C-3)\ $\chi_{_{\mathbf{Z}_+^d}}(\cdot)$ is the indicator of $\mathbf{Z}_+^d$.
\par
(C-4)\ $\emph{\textbf{0}}=(0,\cdots,0)$, $\emph{\textbf{1}}=(1,\cdots,1)$, $\emph{\textbf{e}}_r=(0,\cdots,1_r,\cdots,0)$ are vectors in $[0,1]^d$.
\par
(C-5)\ $T_{rs}(\emph{\textbf{i}})=\emph{\textbf{i}}-\emph{\textbf{e}}_r
+\emph{\textbf{e}}_s$ for $r,s\in \{1,2,\cdots,d\}$ and $\emph{\textbf{i}}\in \mathbf{Z}_+^d$ with $i_r>0$.
\par
(C-6)\ Denote $\emph{\textbf{x}}^{\emph{\textbf{j}}}:=\prod\limits_{r=1}^dx_r^{j_r}$ for any $\emph{\textbf{x}}=(x_1,\cdots,x_d)$ with $x_r\geq 0\ (r=1,\cdots,d)$ and for $\emph{\textbf{j}}=(j_1,\cdots,j_d)\in \mathbf{Z}_+^d$.
\par
By the above description in (i)-(iv), the stopped $M^X/M/1$-queuing network with $d$ nodes satisfies the following conditions:
\par
{\rm{(a)}}\ the state space is $\mathbf{Z}_+^d$;
\par
{\rm{(b)}}\ its generator $Q=(q_{\emph{\textbf{ij}}}:\emph{\textbf{i}},\emph{\textbf{j}}\in \mathbf{Z}_+^d)$ satisfies
\begin{eqnarray}\label{eq1.2}
q_{\emph{\textbf{i}}\emph{\textbf{j}}}
   =\begin{cases}
    \lambda_{rl},
    &  \mbox{if}\  \emph{\textbf{i}}\neq \emph{\textbf{0}}, i_r>0,\ \emph{\textbf{j}}=\emph{\textbf{i}}+l\emph{\textbf{e}}_r,\\
    \mu_{r}\gamma_{_{r0}},
    &  \mbox{if}\ i_r>0, \emph{\textbf{j}}=\emph{\textbf{i}}
    -\emph{\textbf{e}}_r,\\
   \mu_{r}\gamma_{_{rs}},
    & \mbox{if}\ i_r>0, \emph{\textbf{j}}=T_{rs}(\emph{\textbf{i}}),\\
   -\sum\limits_{r=1}^d(\lambda_r+\mu_r)\cdot \chi_{_{\{i_r>0\}}},
    &  \mbox{if}\ \emph{\textbf{j}}=\emph{\textbf{i}}\neq \emph{\textbf{0}},\\
     0,              & \mbox{otherwise}.
\end{cases}
\end{eqnarray}
\par
In the following, we assume that $\gamma_{rr}=0\ (1\leq r\leq d)$, $\gamma_{r0}>0$ for some $r\in \{1,2,\cdots,d\}$ and the routing matrix $\Gamma=(\gamma_{rs}:1\leq r,s\leq d)$ is irreducible, which guarantees that $\mathcal{C}:=\mathbf{Z}^d_+\setminus \{\emph{\textbf{0}}\}$ is irreducible.
\par
For the $q$-matrix $Q$ given in (\ref{eq1.2}), by Theorem~2.2.2 of Anderson~\cite{And91}, we know that $Q$ determines exactly one Markov process, i.e., the Feller minimal $Q$-process. Let $\{X(t):t\geq 0\}$ be the $Q$-process and $P(t)=(p_{\emph{\textbf{i}}\emph{\textbf{j}}}(t):
\emph{\textbf{i}},\emph{\textbf{j}}\in\mathbf{Z}^d_+)$ be the transition probability function of $\{X(t):t\geq 0\}$.
\par
It is well known that the decay parameter and quasi-stationary
distributions are closely linked with the so-called
$\mu$-subinvariant/invariant measures and
$\mu$-subinvariant/invariant vectors. An elementary but detailed
discussion of this theory can be seen in
Anderson~\cite{And91}. For convenience, we briefly repeat
these definitions as follows:
\par
\begin{definition}\label{def1.1}
Let $Q=(q_{\emph{\textbf{ij}}}:\emph{\textbf{i}},\emph{\textbf{j}}\in \mathbf{Z}_+^d)$ be the $q$-matrix given in (\ref{eq1.2}) and $\mathcal{C}$ be a communicating class. Assume that $\mu\geq 0$.
A set $(m_{\emph{\textbf{i}}}:\emph{\textbf{i}}\in \mathcal{C})$ of positive numbers is
called a $\mu$-subinvariant measure for $Q$ on $\mathcal{C}$ if
\begin{eqnarray}\label{eq1.3}
 \sum_{\emph{\textbf{i}}\in C}m_{\emph{\textbf{i}}}q_{\emph{\textbf{ij}}}\leq -\mu m_{\emph{\textbf{j}}},\ \ \ \ \emph{\textbf{j}}\in \mathcal{C}.
\end{eqnarray}
If the equality holds in (\ref{eq1.3}), then $(m_{\emph{\textbf{i}}}:\emph{\textbf{i}}\in \mathcal{C})$ is
called a $\mu$-invariant measure for $Q$ on $\mathcal{C}$.
\end{definition}
\par
\begin{definition}\label{def1.2}
Let $Q=(q_{\emph{\textbf{ij}}}:\emph{\textbf{i}},\emph{\textbf{j}}\in \mathbf{Z}_+^d)$ be the $q$-matrix given in (\ref{eq1.2}), $P(t)=(p_{\emph{\textbf{ij}}}(t):\emph{\textbf{i}},\emph{\textbf{j}}\in \mathbf{Z}_+^d)$ be the $Q$-process and $\mathcal{C}$ be a communicating class. Assume that $(m_{\emph{\textbf{i}}}:\emph{\textbf{i}}\in \mathcal{C})$ is a probability distribution over $\mathcal{C}$.
Denote $p_{\emph{\textbf{j}}}(t)=\sum\limits_{\emph{\textbf{i}}\in C}m_{\emph{\textbf{i}}}p_{\emph{\textbf{ij}}}(t)$ for
$\emph{\textbf{j}}\in \mathcal{C}, t\geq 0$. If
\begin{eqnarray}\label{eq1.4}
  \frac{p_{\emph{\textbf{j}}}(t)}{\sum\limits_{\emph{\textbf{i}}\in C}p_{\emph{\textbf{j}}}(t)}= m_{\emph{\textbf{j}}},\ \ \ \ \emph{\textbf{j}}\in \mathcal{C}, t>0,
\end{eqnarray}
then $(m_{\emph{\textbf{i}}}:\emph{\textbf{i}}\in \mathcal{C})$ is called a quasi-stationary distribution.
\end{definition}
\par
Li and Chen~\cite{LC08,LC13} considered the decay properties of stopped Markovian bulk-arrival queues and Markovian bulk-arrival queues with control at idle time, respectively. Li and Wang~\cite{2012-LiWang} discussed the decay property of $n$-type branching
processes. Chen, Li, Wu and Zhang~\cite{CLWZ21} studied the decay parameter for general stopped Markovian
bulk-arrival and bulk-service queues. The main purpose of this paper is to investigate the
decay properties of stopped $d$-dimensional Markov queueing models. If $d=1$, then it becomes the model considered in Li and Chen~\cite{LC08}.
Therefore, we assume $d\geq 2$ through out this paper. Different from the one-dimensional case, when a customer at one node has finished his/her service, he/she can enter another node and make the queue length at another node changed.
Therefore, the methods used in Li and Chen~\cite{LC13} and Chen, Li, Wu and Zhang~\cite{CLWZ21} are not applicable (see Theorems~\ref{th3.2}-\ref{th3.3}) and
some new approaches should be used in the current situation, which is the main contribution of this paper.
\par
The structure of this paper is organized as follows. Some preliminary results are
firstly establish in Section 2. In Section 3, the exact value of decay parameter is obtained. The $\lambda_{\mathcal{C}}$-invariant measure and quasi-stationary distribution are discussed in
Section 4.

\par
\vspace{5mm}
 \setcounter{section}{2}
 \setcounter{definition}{0}
 \setcounter{equation}{0}
 \setcounter{theorem}{0}
 \setcounter{lemma}{0}
 \setcounter{corollary}{0}
\noindent {\large \bf 2. Preliminaries}
 \vspace{3mm}
\par
In order to discuss the decay property of stopped $d$-dimensional Markov queuing models with $d\geq 2$, we make some preliminaries regarding our model in this section. Since $Q$ is determined by the sequences $\{\lambda_{rk}:k\geq 1, r=1,\cdots,d\}$, $\{\gamma_{rs}:1\leq r\leq d, 0\leq s\leq d\}$ and $\{\mu_r:1\leq r\leq d\}$, we define the following functions as
\begin{eqnarray*}
&&\Lambda_r(x_r)=\sum_{k=1}^{\infty}\lambda_{rk}x_r^{k},\ \ \ r=1,\cdots,d\\
&& \Gamma_r(\emph{\textbf{x}})=\mu_r(\gamma_{r0}
+\sum_{s=1}^{d}\gamma_{rs}x_s),\ \ \ r=1,\cdots,d
\end{eqnarray*}
and denote
\begin{eqnarray*}
B_r(\emph{\textbf{x}}):=x_r[\Lambda_r(x_r)-\lambda_r]
+\Gamma_r(\emph{\textbf{x}})-\mu_rx_r, \ \ r=1,\cdots,d,
\end{eqnarray*}
where $\emph{\textbf{x}}=(x_1,\cdots,x_d)$. Let
\begin{eqnarray*}
\theta_r=\frac{1}{\limsup_{n\rightarrow \infty}\sqrt[n]{\lambda_{rn}}}
\end{eqnarray*}
be the convergence radius of $\Lambda_r(x_r)$ for each $r=1,\cdots, d$. It is obvious that $\theta_r\geq 1\ (r=1,\cdots,d)$ and therefore,
$\{B_r(\emph{\textbf{x}}):r=1,\cdots,d\}$ are well-defined at least on $\prod\limits_{s=1}^d[0,\theta_s]$.
\par
For convenience of notation, let
\begin{eqnarray}\label{eq2.1}
b^{(r)}_{\emph{\textbf{j}}}=
\begin{cases}
\mu_r\gamma_{_{r0}},\ & \emph{\textbf{j}}=\emph{\textbf{0}},\\
-\mu_r-\lambda_r,\ & \emph{\textbf{j}}=\emph{\textbf{e}}_r,\\
\mu_r\gamma_{{rs}},\ & \emph{\textbf{j}}=\emph{\textbf{e}}_s, s\neq r,\\
\lambda_{rk}, \ & \emph{\textbf{j}}=(k+1)\emph{\textbf{e}}_r, k\geq 1,\\
0,\ & \text{otherwise}.
\end{cases}
\end{eqnarray}
Then, $(q_{\emph{\textbf{ij}}};\emph{\textbf{i}},\emph{\textbf{j}}\in \mathbf{Z}_+^d)$ and $\{B_r(\emph{\textbf{x}}):1\leq r\leq d\}$ can be rewritten as
\begin{eqnarray*}
q_{\emph{\textbf{ij}}}=
\begin{cases}
\sum\limits_{r=1}^db^{(r)}_{\emph{\textbf{j}}
-\emph{\textbf{i}}+\emph{\textbf{e}}_r}\chi_{_{\{i_r>0,
\emph{\textbf{j}}-\emph{\textbf{i}}+\emph{\textbf{e}}_r\in \mathbf{Z}_+^d\}}},\ & \emph{\textbf{i}}\neq \emph{\textbf{0}}, \emph{\textbf{j}}\in \mathbf{Z}_+^d,\\
0,\ & \text{otherwise}
\end{cases}
\end{eqnarray*}
and
\begin{eqnarray*}
B_r(\emph{\textbf{x}})=\sum\limits_{\emph{\textbf{j}}\in \mathbf{Z}_+^d}b^{(r)}_{\emph{\textbf{j}}}\emph{\textbf{x}}
^{\emph{\textbf{j}}}
\end{eqnarray*}
respectively.
\par
The following lemma shows the basic property of the transition probability function of $\{X(t):t\geq 0\}$.
\par
\begin{lemma}
\label{le2.1}
 Let $Q=(q_{\textbf{ij}}:\textbf{i},\textbf{j}\in \mathbf{Z}_+^d)$ be the $q$-matrix given in $(\ref{eq1.2})$, $P(t)=(p_{\textbf{ij}}(t):\textbf{i},\textbf{j}\in \mathbf{Z}_+^d)$ and $\Phi(\lambda)=(\phi_{\textbf{ij}}(\lambda):\textbf{i},\textbf{j}\in \mathbf{Z}_+^d)$ be the $Q$-function and
$Q$-resolvent $($i.e., $\phi_{\textbf{ij}}(\lambda):=\int_0^{\infty}e^{-\lambda t}p_{\textbf{ij}}(t)dt$$)$, respectively.
Then for any $\textbf{i}\in \mathbf{Z}_+^d$,
\begin{eqnarray}\label{eq2.2}
p'_{\textbf{i0}}(t)+\frac{\partial F_{\textbf{i}}(t,\textbf{x})}{\partial t}=
\sum_{r=1}^dB_r(\textbf{x})\cdot F^{(r)}_{\textbf{i}}(t,\textbf{x}),
\end{eqnarray}
or in resolvent version
\begin{eqnarray}\label{eq2.3}
\lambda \phi_{\textbf{i0}}(\lambda)+ \lambda\Phi_{\textbf{i}}(\lambda,\textbf{x})-\textbf{x}^{\textbf{i}}=
\sum_{r=1}^dB_r(\textbf{x})\cdot\Phi^{(r)}_{\textbf{i}}(\lambda,\textbf{x}),
\end{eqnarray}
where $F_{\textbf{i}}(t,\textbf{x})=\sum\limits_{\textbf{j}\in
\mathcal{C}} p_{\textbf{ij}}(t)\cdot\textbf{x}^{\textbf{j}}$, $F^{(r)}_{\textbf{i}}(t,\textbf{x})=\sum\limits_{\textbf{j}\in
\mathcal{C}_r^+} p_{\textbf{ij}}(t)\cdot\textbf{x}^{\textbf{j}-e_r}$, $\Phi_{\textbf{i}}(\lambda,\textbf{x})=\sum\limits_{\textbf{j}\in
\mathcal{C}} \phi_{\textbf{ij}}(\lambda)\cdot\textbf{x}^{\textbf{j}}$ and $\Phi^{(k)}_{\textbf{i}}(\lambda,\textbf{x})
=\sum\limits_{\textbf{j}\in
\mathcal{C}_r^+} \phi_{\textbf{ij}}(\lambda)\cdot\textbf{x}^{\textbf{j}-e_r}$ with $\mathcal{C}^+_r=\{\textbf{i}\in \mathcal{C}: i_r>0\}$.
\end{lemma}
\begin{proof}
By the Kolmogorov forward equations, we know that for any $\emph{\textbf{i}},\emph{\textbf{j}}\in \mathbf{Z}_{+}^d$,
\begin{eqnarray*}
p'_{\emph{\textbf{i0}}}(t)&=&\sum\limits_{r=1}^d
p_{\emph{\textbf{i}}\emph{\textbf{e}}_r}(t)
b^{(r)}_{\emph{\textbf{0}}}\\
p'_{\emph{\textbf{ij}}}(t)
&=&\sum_{\emph{\textbf{l}}\in \mathcal{C}}p_{\emph{\textbf{i}}
\emph{\textbf{l}}}(t)\sum_{r=1}^db^{(r)}_{\emph{\textbf{j}}
-\emph{\textbf{l}}+\emph{\textbf{e}}_r}
\chi_{_{\{l_r>0,\emph{\textbf{j}}
-\emph{\textbf{l}}+\emph{\textbf{e}}_r\in \mathbf{Z}_+^d\}}}\\
&=&\sum_{r=1}^d\sum_{\emph{\textbf{l}}\in \mathcal{C}_r^+}p_{\emph{\textbf{i}}
\emph{\textbf{l}}}(t)b^{(r)}_{\emph{\textbf{j}}
-\emph{\textbf{l}}+\emph{\textbf{e}}_r}
\chi_{_{\{\emph{\textbf{j}}
-\emph{\textbf{l}}+\emph{\textbf{e}}_r\in \mathbf{Z}_+^d\}}},\ \ \emph{\textbf{j}}\in \mathcal{C}.
\end{eqnarray*}
Multiplying $\emph{\textbf{x}}^{\emph{\textbf{j}}}$ on both sides of the above equality and then summing on $\emph{\textbf{j}}\in \mathbf{Z}_+^d$ yields
\begin{eqnarray*}
\sum_{\emph{\textbf{j}}\in \mathbf{Z}^d_+}p'_{\emph{\textbf{ij}}}(t)\emph{\textbf{x}}^{\emph{\textbf{j}}}
=\sum_{r=1}^dB_r(\emph{\textbf{x}})\sum_{\emph{\textbf{j}}\in \mathcal{C}_r^+}
p_{\emph{\textbf{ij}}}(t)\emph{\textbf{x}}^{\emph{\textbf{j}}
-\emph{\textbf{e}}_r}
\end{eqnarray*}
Hence, (\ref{eq2.2}) is proved. Taking Laplace transform on both sides of (\ref{eq2.2})
immediately yields (\ref{eq2.3}).\hfill$\Box$
\end{proof}
\par
Denote $\emph{\textbf{B}}(\emph{\textbf{x}})=(B_1(\emph{\textbf{x}}),\cdots,
B_d(\emph{\textbf{x}}))$ and let
\begin{eqnarray*}
B_{rs}(\emph{\textbf{x}})=\frac{\partial B_r(\emph{\textbf{x}})}{\partial x_s},\quad g_{rs}(\emph{\textbf{x}})=\delta_{rs}+
\frac{B_{rs}(\emph{\textbf{x}})}{\mu_r+\lambda_r},\quad r,s=1,\cdots,d.
\end{eqnarray*}
Obviously, the functions $B_{rs}(\emph{\textbf{x}})$ and $g_{rs}(\emph{\textbf{x}})$ are related to the transition probabilities of the jump chain. The matrices $(B_{rs}(\emph{\textbf{x}}))$ and $(g_{rs}(\emph{\textbf{x}}))$ are denoted by $\emph{\textbf{B}}'(\emph{\textbf{x}})$ and $G(\emph{\textbf{x}})$, respectively. Since
\begin{eqnarray*}
g_{rs}(1)=
\begin{cases}
(\mu_r+\lambda_r)^{-1}\sum\limits_{k=1}^{\infty}(k+1)
\lambda_{rk}, \ & r=s,\\
(\mu_r+\lambda_r)^{-1}\mu_r\gamma_{_{rs}},\ & r\neq s,
\end{cases}
\end{eqnarray*}
it is easy to see that $G(\emph{\textbf{1}})$ is strictly positive, i.e., there exists an integer $N$ such that $[G(\emph{\textbf{1}})]^N>0$.
\par
Let $\rho(\emph{\textbf{x}})$ denote the maximal eigenvalue of $\emph{\textbf{B}}'(\emph{\textbf{x}})$. The following lemma presents a property of $\emph{\textbf{B}}(\emph{\textbf{x}})$.
\par
\begin{lemma}\label{le2.2}
The system of equations
\begin{eqnarray}\label{eq2.4}
\textbf{B}(\textbf{x})=\textbf{0}
\end{eqnarray}
has at most two solutions in $[0,1]^d$. Let $\textbf{q}=(q_1,\cdots,q_d)$ denote the smallest
nonnegative solution of $(\ref{eq2.4})$. Then,
\par
{\rm (i)}\  $q_r$ is the extinction probability when the process starts at state $\textbf{e}_r\ (r=1,\cdots,d)$.
Moreover, if $\rho(\textbf{1})\leq 0$, then $\textbf{q}=\textbf{1}$; while
if $\rho(\textbf{1})>0$, then $\textbf{q}<\textbf{1}$, i.e.,
$q_1,\cdots,q_d<1$.
\par
{\rm (ii)}\  $\rho(\textbf{q})\leq 0$.
\end{lemma}
\par
\begin{proof}
Since $G(\emph{\textbf{1}})$ is positively regular, $\Gamma=(\gamma_{rs}:1\leq r,s\leq d)$ is irreducible and $\lambda_r>0\ (r=1,\cdots,d)$, (i)-(ii) follow from Theorem~V.3.2 in Athreya and Ney~\cite{Ath72}.  \hfill $\Box$
\end{proof}
\par
In addition, the following lemma shows a property of $\rho(\emph{\textbf{x}})$.
\par
\begin{lemma}\label{le2.3}
\par
{\rm{(i)}}\ $\rho(\textbf{x})$ is continuous on $\prod\limits_{r=1}^d[0,\theta_r)$. Moreover, for any $\textbf{x}<\tilde{\textbf{x}}$ {\rm{(}}i.e., $x_r<\tilde{x}_r$ for all $r=1,\cdots,d${\rm{)}}, $\rho(\textbf{x})<\rho({\tilde{\textbf{x}}})$.
\par
{\rm (ii)}\ For any $\textbf{x}\in [0,1]^d$, there exist positive vectors $\textbf{v}(\textbf{x})$ and $\textbf{u}(\textbf{x})$ such that
\begin{eqnarray*}
\textbf{B}'(\textbf{x})\textbf{v}^T(\textbf{x})
=\rho(\textbf{x})\textbf{v}^T(\textbf{x})
\end{eqnarray*}
and
\begin{eqnarray*}
\textbf{u}(\textbf{x})\textbf{B}'(\textbf{x})
=\rho(\textbf{x})\textbf{u}(\textbf{x}).
\end{eqnarray*}
\end{lemma}
\par
\begin{proof}
The characteristic polynomial of $\emph{\textbf{B}}'(\emph{\textbf{x}})$ can be expressed as
\begin{eqnarray*}
f(\lambda;\emph{\textbf{x}})=\lambda^d+A_1(\emph{\textbf{x}})\lambda^{d-1}
+\cdots+A_d(\emph{\textbf{x}}),
\end{eqnarray*}
where all the functions $A_r(\emph{\textbf{x}})\ (r=1,\cdots,d)$ are continuous in $\emph{\textbf{x}}\in \prod\limits_{r=1}^d[0,\theta_r)$ because $B_{rs}(\emph{\textbf{x}})\ (r,s=1,\cdots,d)$ are continuous. For any $\emph{\textbf{x}}\rightarrow \tilde{\emph{\textbf{x}}}$, if $f(\lambda;\tilde{\emph{\textbf{x}}})>0$, then it follows from the continuity of $f(\lambda;\emph{\textbf{x}})$ in $\emph{\textbf{x}}$ that $f(\lambda;\emph{\textbf{x}})>0$ for $\emph{\textbf{x}}$ closing to $\tilde{\emph{\textbf{x}}}$. Therefore, $\liminf\limits_{\emph{\textbf{x}}\rightarrow \tilde{\emph{\textbf{x}}}}\rho(\emph{\textbf{x}})\geq \rho(\tilde{\emph{\textbf{x}}})$. Similarly, if $f(\lambda;\tilde{\emph{\textbf{x}}})<0$, then  $f(\lambda;\emph{\textbf{x}})<0$ for $\emph{\textbf{x}}$ closing to $\tilde{\emph{\textbf{x}}}$. Therefore, $\limsup\limits_{\emph{\textbf{x}}\rightarrow \tilde{\emph{\textbf{x}}}}\rho(\emph{\textbf{x}})\leq \rho(\tilde{\emph{\textbf{x}}})$. Thus, $\rho(\emph{\textbf{x}})$ is continuous. The monotone property of $\rho(\emph{\textbf{x}})$ follows from Lemma~2.6 of Li and Wang~\cite{2012-LiWang}. (i) is proved.
\par
(ii) follows from Theorem~V.2.1 in Athreya and Ney~\cite{Ath72}. The proof is complete. \hfill $\Box$
\end{proof}

\par
\vspace{5mm}
 \setcounter{section}{3}
 \setcounter{equation}{0}
 \setcounter{theorem}{0}
 \setcounter{lemma}{0}
 \setcounter{corollary}{0}
 \setcounter{remark}{0}
\noindent {\large \bf 3. Decay parameter}
 \vspace{3mm}
 \par
Having made some preliminaries in the previous section, we now consider the decay parameter
$\lambda_{\mathcal{C}}$ of $P(t)$ on $\mathcal{C}=\mathbf{Z}^d_+\setminus\{\emph{\textbf{0}}\}$.
\par
For any $\lambda\in (-\infty,\infty)$, consider the system of inequalities:
\begin{eqnarray}\label{eq3.1}
\emph{\textbf{B}}(\emph{\textbf{x}})+\lambda \emph{\textbf{x}}\leq \emph{\textbf{0}}.
\end{eqnarray}
\par
Let
\begin{eqnarray}\label{eq3.2}
\lambda_*=\sup\{\lambda:(\ref{eq3.1})\ has\ a\ solution\ in\ [0,\infty)^d\}.
\end{eqnarray}
\par
The following lemma proves that $\lambda_*$ is finite and shows a property of the equation $\emph{\textbf{B}}(\emph{\textbf{x}})+\lambda \emph{\textbf{x}}=\emph{\textbf{0}}$ for $\lambda\leq \lambda_*$.
\par
\begin{lemma}\label{le3.1}
{\rm{(i)}}\ $0\leq\lambda_*\leq \min\{\lambda_r+\mu_r:r=1,\cdots,d\}$.
\par
{\rm{(ii)}}\ The system of equations
\begin{eqnarray}\label{eq3.3}
\textbf{B}(\textbf{x})+\lambda_* \textbf{x}=\textbf{0}
\end{eqnarray}
has exactly one solution $\textbf{q}_*=(q_{*1},\cdots,q_{*d})$ on $\prod\limits_{r=1}^d(0,\theta_r]$.
\par
{\rm{(iii)}}\ For any $\lambda<\lambda_*$, the system of equations
\begin{eqnarray}\label{eq3.4}
\textbf{B}(\textbf{x})+\lambda \textbf{x}=\textbf{0}
\end{eqnarray}
has exactly one solution on $\prod\limits_{r=1}^d(0,q_{*r}]$.
\end{lemma}
\par
\begin{proof} Since $\emph{\textbf{x}}=\emph{\textbf{1}}$ is a solution of (\ref{eq3.1}) for $\lambda=0$, we know that $\lambda_*\geq 0$. On the other hand, for any $\lambda$ such that (\ref{eq3.1}) has a solution $\bar{\emph{\textbf{x}}}\in [0,\infty)^d$,
\begin{eqnarray*}
\lambda\leq-\frac{B_r(\bar{\emph{\textbf{x}}})}{x_r}\leq \lambda_r+\mu_r,\ \ r=1,\cdots,d.
\end{eqnarray*}
Hence, $\lambda_*\leq \min\{\lambda_r+\mu_r:r=1,\cdots,d\}$. (i) is proved.
\par
Next prove (ii). Since $\gamma_{r0}>0$ for all $r\in \{1,2,\cdots,d\}$, by the definition of $\lambda_*$ and the continuity of $B_r(\emph{\textbf{x}})$, we know that the system of inequalities
\begin{eqnarray}\label{eq3.5}
\emph{\textbf{B}}(\emph{\textbf{x}})+\lambda_* \emph{\textbf{x}}\leq \emph{\textbf{0}}
\end{eqnarray}
has solution on $\prod\limits_{r=1}^d(0,\theta_r]$. Next, we prove that (\ref{eq3.5})
has exactly one solution on $\prod\limits_{r=1}^d(0,\theta_r]$.
Indeed, suppose $\tilde{\emph{\textbf{x}}}=(\tilde{x}_1,\cdots,\tilde{x}_d)$ and $\hat{\emph{\textbf{x}}}=(\hat{x}_1,\cdots,\hat{x}_d)$ are two different solutions of (\ref{eq3.5}). Denote $f_r(\emph{\textbf{x}})=B_r(\emph{\textbf{x}}))+\lambda_*x_r,\ (r=1,\cdots,d)$ and let $f''_r(\emph{\textbf{x}})$ be the Hessian matrix of $f_r(\emph{\textbf{x}})$, i.e., $f''_r(\emph{\textbf{x}})=(\frac{\partial^2f_r(\emph{\textbf{x}})}{\partial x_k\partial x_l}:k,l=1,\cdots,d)$. Then, for any $r$, the matrix $f''_r(\emph{\textbf{x}})=diag(\sum\limits_{k=1}^{\infty}(k+1)k
\lambda_{1k}x_1^{k-1},\cdots,
\sum\limits_{k=1}^{\infty}(k+1)k\lambda_{dk}x_d^{k-1})$, which is positive definite for $\emph{\textbf{x}}> \emph{\textbf{0}}$. Therefore, $f_r(\emph{\textbf{x}})$ is convex on $\prod\limits_{s=1}^d(0,\theta_s]$. By the property of convex function, we know that for any $\delta\in (0,1)$ and $\emph{\textbf{z}}=(1-\delta)\tilde{\emph{\textbf{x}}}
+\delta\hat{\emph{\textbf{x}}}$,
\begin{eqnarray*}
f_r(\emph{\textbf{z}})<(1-\delta)f_r(\tilde{\emph{\textbf{x}}})+\delta f_r(\hat{\emph{\textbf{x}}})\leq 0.
\end{eqnarray*}
Fix $\delta_0\in (0,1)$ and take $\bar{\emph{\textbf{z}}}=(1-\delta_0)\tilde{\emph{\textbf{x}}}
+\delta_0\hat{\emph{\textbf{x}}}$. Then
\begin{eqnarray*}
f_r(\bar{\emph{\textbf{z}}})=B_r(\bar{\emph{\textbf{z}}})
+\lambda_*\bar{z}_r<0, \ \ r=1,\cdots,d.
\end{eqnarray*}
Take
\begin{eqnarray*}
\bar{\lambda}=\min\{-\frac{B_r(\bar{\emph{\textbf{z}}})}{\bar{z}_r}:r=1,\cdots,d\}.
\end{eqnarray*}
Then, $\bar{\lambda}>\lambda_*$ and $B_r(\bar{\emph{\textbf{z}}})+\bar{\lambda}\bar{z}_r\leq B_r(\bar{\emph{\textbf{z}}})-\frac{B_r(\bar{\emph{\textbf{z}}})}
{\bar{z}_r}\cdot \bar{z}_r=0$ for $r=1,\cdots,d$, which contradicts the definition of $\lambda_*$. Therefore, (\ref{eq3.5}) has exactly one solution on $\prod\limits_{r=1}^d(0,\theta_r]$. Hence, by Li and Wang~\cite{2012-LiWang}, we further know that (\ref{eq3.3}) has exact one solution on $\prod\limits_{r=1}^d(0,\theta_r]$.
\par
Now prove (iii). Since $\emph{\textbf{q}}_*=(q_{*1},\cdots,q_{*d})$ is a solution of (\ref{eq3.3}), we see that $\emph{\textbf{B}}(\emph{\textbf{q}}_*)+\lambda \emph{\textbf{q}}_*< \emph{\textbf{0}}$. By Li and Wang~\cite{2012-LiWang}, (\ref{eq3.4}) has exactly one solution on $\prod\limits_{r=1}^d(0,q_{*r}]$.
 The proof is complete. \hfill $\Box$
\end{proof}
\par
By Lemma~\ref{le3.1}, we use $\emph{\textbf{q}}_*=(q_{*r}:r=1,\cdots,d)$ to denote the unique solution of (\ref{eq3.3}) and use $\emph{\textbf{q}}(\lambda)=(q_1(\lambda),\cdots,q_d(\lambda))$ to denote the unique solution of (\ref{eq3.4}) on $\prod\limits_{r=1}^d[0,q_{*r}]$ for $\lambda\in [0,\lambda_*]$ in the following.
The following theorem reveals the relationship of $\rho(\emph{\textbf{q}}(\lambda))$ and $\lambda\in [0,\lambda_*]$.
\par
\begin{theorem}\label{th3.1}
{\rm{(i)}}\ For any $\lambda\in [0,\lambda_*]$, let $\textbf{q}(\lambda)=(q_1(\lambda),\cdots,q_d(\lambda))$ be the unique solution of $\textbf{B}(\textbf{x})+\lambda \textbf{x}=\textbf{0}$ on $\prod\limits_{r=1}^d[0,q_{*r}]$. We have $\rho(\textbf{q}(\lambda))\leq -\lambda$.
\par
{\rm{(ii)}}\
If $\textbf{q}_*\in \prod\limits_{r=1}^d(0,\theta_r)$, then $\rho(\textbf{q}_*)=-\lambda_*$.
\par
{\rm{(iii)}}\ If $\rho(\textbf{1})=0$, then $\lambda_*=0$ and $\textbf{q}_*=\textbf{1}$.
\end{theorem}
\begin{proof}
First prove (i). Let $\lambda\in [0,\lambda_*]$ and $\emph{\textbf{q}}(\lambda)=(q_1(\lambda),\cdots,q_d(\lambda))$ be the unique solution of $\emph{\textbf{B}}(\emph{\textbf{x}})+\lambda \emph{\textbf{x}}=\emph{\textbf{0}}$ on $\prod\limits_{r=1}^d(0,q_{*r}]$.
Consider
\begin{eqnarray*}
\tilde{B}_r(\emph{\textbf{x}})=B_r(\emph{\textbf{q}}(\lambda)\otimes \emph{\textbf{x}})+\lambda q_r(\lambda)x_r,\ \ r=1,\cdots,d,
\end{eqnarray*}
where $\emph{\textbf{q}}(\lambda)\otimes \emph{\textbf{x}}=(q_1(\lambda)x_1,\cdots,q_d(\lambda)x_d)$.
It is easy to see that
\begin{eqnarray*}
\tilde{\emph{\textbf{B}}}(\emph{\textbf{x}})
=(\tilde{B}_1(\emph{\textbf{x}}),\cdots,\tilde{B}_d(\emph{\textbf{x}}))
=\emph{\textbf{0}}
\end{eqnarray*}
has exactly one solution on $[0,1]^d$, that is $\emph{\textbf{1}}$. Let $\tilde{\rho}(\emph{\textbf{1}})$ denote the maximal eigenvalue of $\tilde{\emph{\textbf{B}}}'(\emph{\textbf{1}})$. By Lemma~\ref{le2.2}, $\tilde{\rho}(\emph{\textbf{1}})\leq 0$. On the other hand, one can check that $\tilde{\emph{\textbf{B}}}'(\emph{\textbf{1}})=(\emph{\textbf{B}}'(\emph{\textbf{q}}(\lambda))+\lambda I)\cdot {\rm{diag}}(q_1(\lambda),\cdots,q_d(\lambda))$. Hence, by Lemma 2.4 of Li and Wang~\cite{2012-LiWang}, we know that $\rho(\emph{\textbf{q}}(\lambda))+\lambda\leq 0$, i.e., $\rho(\emph{\textbf{q}}(\lambda))\leq -\lambda$. (i) is proved.
\par
Next prove (ii). By Lemma~2.6 of Li and Wang~\cite{2012-LiWang}, there exists a positive eigenvector $\emph{\textbf{v}}^*=(v^*_1,\cdots,v^*_d)$ such that
\begin{eqnarray*}
\sum_{s=1}^dB_{rs}(\emph{\textbf{q}}_*)v^*_s=
\rho(\emph{\textbf{q}}_*)v^*_r, \ \ r=1,\cdots,d.
\end{eqnarray*}
Consider $g_r(\varepsilon):=B_r(\emph{\textbf{q}}_{*}+\varepsilon \emph{\textbf{v}}^*)+\lambda_*(q_{*r}+\varepsilon v^*_r)\ (1\leq r\leq d)$ for $\varepsilon\in [0,\min\limits_{1\leq s\leq d}\frac{\theta_s-q_{s*}}{v^*_s})$. By Taylor's mean value theorem,
\begin{eqnarray*}
g_r(\varepsilon)
&=&g_r(0)+g'_r(0)\varepsilon+\frac{1}{2}g''_r(\delta \varepsilon)\varepsilon^2\\
&=&\sum_{s=1}^dB_{rs}(\emph{\textbf{q}}_{*})v^*_s\varepsilon
+\lambda_* v^*_r\varepsilon+\frac{1}{2}g''_r(\delta \varepsilon)\varepsilon^2\\
&=&[\rho(\emph{\textbf{q}}_*)+\lambda_*] v^*_r\varepsilon
+\frac{1}{2}g''_r(\delta \varepsilon)\varepsilon^2,\ \ r=1,\cdots,d,
\end{eqnarray*}
where $\delta\in (0,1)$.
Since $g_r(\cdot)\in C^{\infty}[0,\min\limits_{1\leq s\leq d}\frac{\theta_s-q_{s*}}{v^*_s})$, if $\rho(\emph{\textbf{q}}_*)<-\lambda_*$, then by the above equalities, there exists $\varepsilon_0>0$ such that
$$
g_r(\varepsilon)=B_r(\emph{\textbf{q}}_{*}+\varepsilon \emph{\textbf{v}}^*)+\lambda_*(q_{*r}+\varepsilon v^*_r)<0,\ \ r=1,\cdots, d
$$
for $\varepsilon\in [0,\varepsilon_0)$. Hence, there exists $\tilde{\varepsilon}>0$ such that
\begin{eqnarray*}
B_r(\emph{\textbf{q}}_{*}+\tilde{\varepsilon} \emph{\textbf{v}}^*)+\lambda_*(q_{*r}+\tilde{\varepsilon} v^*_r)<0,\ \ r=1,\cdots,d.
\end{eqnarray*}
Let $\tilde{\lambda}=-\max\{\frac{B_r(\emph{\textbf{q}}_{*}
+\tilde{\varepsilon} \emph{\textbf{v}}^*)+\lambda_*(q_{*r}+\tilde{\varepsilon} v^*_r)}{q_{*r}+\tilde{\varepsilon}v^*_r}:1\leq r\leq d\}$. Then, $\tilde{\lambda}>\lambda_*$ and
\begin{eqnarray*}
B_r(\emph{\textbf{q}}_{*}+\tilde{\varepsilon} \emph{\textbf{v}}^*)+\tilde{\lambda}(q_{*r}+\tilde{\varepsilon}v^*_r)\leq 0,\ \ r=1,\cdots,d,
\end{eqnarray*}
which contradicts the definition of $\lambda_*$. Therefore, by (i), $\rho(\emph{\textbf{q}}_*)=-\lambda_*$.
\par
Now prove (iii). Suppose that $\rho(\emph{\textbf{1}})=0$ but $\lambda_*>0$. Denote $D=\{r:q_{*r}< 1\}$. Since $\rho(\emph{\textbf{1}})=0$, by Lemma~2.8 in Li and Wang~\cite{2012-LiWang}, we know that $\{1,2,\cdots,d\}\setminus D\neq \emptyset$. If $D\neq \emptyset$, then we can assume $D=\{1,\cdots,\tilde{r}\}$ without loss of generality. Then,
\begin{eqnarray*}
B_{r}(q_{*1},\cdots,q_{*\tilde{r}},1,\cdots,1)\leq B_{r}(\emph{\textbf{q}}_*)<0,\ \ r=1,\cdots,\tilde{r}
\end{eqnarray*}
and
\begin{eqnarray*}
B_{r}(q_{*1},\cdots,q_{*\tilde{r}},1,\cdots,1)\leq B_{r}(\emph{\textbf{1}})\leq 0,\ \ r=\tilde{r}+1,\cdots,d.
\end{eqnarray*}
By Lemma 2.8 in Li and Wang~\cite{2012-LiWang},
\begin{eqnarray*}
\emph{\textbf{B}}(\emph{\textbf{x}})=0
\end{eqnarray*}
has a solution on $\prod\limits_{r=1}^{\tilde{r}}[0,q_{*r}]\times [0,1]^{d-\tilde{r}}$, which contradicts $\rho(\emph{\textbf{1}})=0$. Therefore, we have $\emph{\textbf{q}}_*\in \prod\limits_{r=1}^d(1,\theta_r]$. However,
\begin{eqnarray*}
-\lambda_*q_{*r}=B_r(\emph{\textbf{q}}_*)=\sum\limits_{s=1}^d
B_{rs}(\emph{\textbf{1}})(q_{*s}-1)+\frac{1}{2}\frac{\partial^2B_{r}(\emph{\textbf{a}})}{\partial x_r^2}(q_{*r}-1)^2,\ \ r=1,\cdots,d,
\end{eqnarray*}
where $\emph{\textbf{a}}\in \prod\limits_{s=1}^d[1,q_{*s}]$.
By Lemma~\ref{le2.2}, there exists a positive vector $\emph{\textbf{u}}^*=(u^*_1,\cdots,u^*_d)$ such that
\begin{eqnarray*}
\sum\limits_{r=1}^du^*_rB_{rs}(\emph{\textbf{1}})=0,\ \ s=1,\cdots,d.
\end{eqnarray*}
Therefore,
\begin{eqnarray*}
-\sum\limits_{r=1}^du^*_r\lambda_*q_{*r}=\frac{1}{2}
\sum\limits_{r=1}^du^*_r\frac{\partial^2B_{r}(\emph{\textbf{a}})}{\partial x_r^2}(q_{*r}-1)^2,
\end{eqnarray*}
which is a contradiction. Therefore, $\lambda_*=0$. Hence, $\emph{\textbf{q}}_*=\emph{\textbf{1}}$.
The proof is complete. \hfill $\Box$
\end{proof}
\par
The following theorem gives a lower bound of $\lambda_{\mathcal{C}}$.
\par
\begin{theorem}
\label{th3.2} Suppose that $Q$ is the $q$-matrix given in $(\ref{eq1.2})$ and $P(t)=(p_{\textbf{ij}}(t):\textbf{i},\textbf{j}\in \mathbf{Z}_+^d)$ is the $Q$-function. Then $\lambda_{\mathcal{C}}\geq \lambda_*$. Furthermore, if $\lambda_*=-\rho(\textbf{q}_*)$, then
\begin{eqnarray*}
\lambda_{\mathcal{C}}=\lambda_*.
\end{eqnarray*}
\end{theorem}
\begin{proof}
First prove $\lambda_{\mathcal{C}}\geq \lambda_*$. Let $\emph{\textbf{q}}_*=(q_{*1},\cdots,q_{*d})$ be the nonnegative solution of $(\ref{eq3.3})$. Define
\begin{eqnarray*}
v_{\emph{\textbf{i}}}=\emph{\textbf{q}}_*^{\emph{\textbf{i}}},\ \ \emph{\textbf{i}}=(i_1,\cdots,i_d)\in \mathcal{C}.
\end{eqnarray*}
Denote $D(\emph{\textbf{i}})=\{r:i_r>0\}$. Then $D(\emph{\textbf{i}})\neq \emptyset$ for $\emph{\textbf{i}}=(i_1,\cdots,i_d)\in \mathcal{C}$ and
\begin{eqnarray*}
&&\sum_{\emph{\textbf{j}}\in \mathcal{C}}q_{\emph{\textbf{ij}}}v_{\emph{\textbf{j}}}\\
&=&\sum_{r\in D(\emph{\textbf{i}})}\left(\mu_r[\gamma_{_{r0}}\emph{\textbf{q}}_{*}^
{\emph{\textbf{i}}-\emph{\textbf{e}}_r}
+\sum_{s=1}^d\gamma_{_{rs}}\emph{\textbf{q}}_*^{\emph{\textbf{i}}
-\emph{\textbf{e}}_r+\emph{\textbf{e}}_s}-
\emph{\textbf{q}}_*^{\emph{\textbf{i}}}]
+\sum_{k=1}^{\infty}\lambda_{rk}
\emph{\textbf{q}}_*^{\emph{\textbf{i}}+k\emph{\textbf{e}}_r}-\lambda_r
\emph{\textbf{q}}_*^{\emph{\textbf{i}}}\right)\\
&=&\emph{\textbf{q}}_*^{\emph{\textbf{i}}}\sum_{r\in D(\emph{\textbf{i}})}\left[q_{*r}^{-1}\mu_r(\gamma_{_{r0}}
+\sum_{s=1}^d\gamma_{_{rs}}q_{*s}-q_{*r})
+\sum_{k=1}^{\infty}\lambda_{rk}q_{*r}^{k}
-\lambda_r\right]\\
&=&\emph{\textbf{q}}_*^{\emph{\textbf{i}}}\sum_{r\in D(\emph{\textbf{i}})}\frac{B_r(\emph{\textbf{q}}_*)}{q_{*r}}\\
&\leq & -\lambda_*v_{\emph{\textbf{i}}}.
\end{eqnarray*}
Therefore, $(\emph{\textbf{q}}_*^{\emph{\textbf{j}}}:\emph{\textbf{j}}\in \mathcal{C})$ is a $\lambda_*$-subinvariant vector for $Q$ on $\mathcal{C}$. By the Anderson~\cite{And91} (remarks on p.175), we know that $\lambda_{\mathcal{C}}\geq \lambda_*$.
\par
Now prove the second assertion. By Theorem~\ref{th3.1}, there exists a unique $\emph{\textbf{q}}_*=(q_{1},\cdots,q_{*d})\in \prod\limits_{s=1}^d[0,\theta_s]$ such that $\emph{\textbf{B}}(\emph{\textbf{q}}_*)
+\lambda_*\emph{\textbf{q}}_*=\emph{\textbf{0}}$ and hence by Li and Wang~\cite{2012-LiWang}, we have $\rho(\emph{\textbf{q}}_*)\leq 0$.
\par
{\rm{(a)}}\ First assume that $\rho(\emph{\textbf{1}})=0$. In this case, $\lambda_*=0$ and $\emph{\textbf{q}}_*=\emph{\textbf{1}}$.
Suppose that $\lambda_{\mathcal{C}}>\lambda_*=0$. Then for any $\lambda\in (0,\lambda_{\mathcal{C}})$, we have $\int_0^{\infty}e^{\lambda t}p_{\emph{\textbf{i}}\emph{\textbf{e}}_r}(t)dt<\infty$ for $\emph{\textbf{i}}\neq \emph{\textbf{0}}$ and $r=1,\cdots,d$.
By the proof of Lemma~\ref{le2.1}, $\int_0^{\infty}e^{\lambda t}p'_{\emph{\textbf{i}}\emph{\textbf{0}}}(t)dt<\infty$ and hence $T^{(r)}(\emph{\textbf{x}}):
=\int_0^{\infty}F^{(r)}_{\emph{\textbf{e}}_1}(t,\emph{\textbf{x}})dt
<\infty$ for $\emph{\textbf{x}}\in \prod\limits_{s=1}^d[0,\theta_s)$.
Therefore,
\begin{eqnarray*}
1-x_1
=\sum\limits_{r=1}^dB_r(\emph{\textbf{x}})
\cdot T^{(r)}(\emph{\textbf{x}})
\end{eqnarray*}
and hence
\begin{eqnarray}\label{eq3.6}
-\delta_{1,s}
=\sum\limits_{r=1}^dB_{rs}(\emph{\textbf{x}})
\cdot T^{(r)}(\emph{\textbf{x}})+\sum\limits_{r=1}^dB_{r}(\emph{\textbf{x}})
\cdot \frac{\partial T^{(r)}(\emph{\textbf{x}})}{\partial x_s},\ \ s=1,\cdots,d.
\end{eqnarray}
On the other hand, by Lemma~\ref{le2.3}, there exists a positive vector $\emph{\textbf{v}}(\emph{\textbf{x}})=(v_1(\emph{\textbf{x}}),\cdots,
v_d(\emph{\textbf{x}}))$ such that
\begin{eqnarray*}
\sum_{s=1}^dB_{rs}(\emph{\textbf{x}})\cdot v_s(\emph{\textbf{x}})=0,\ \ r=1,\cdots,d.
\end{eqnarray*}
Multiplying $v_s(\emph{\textbf{x}})$ on the both sides of (\ref{eq3.6}) and then summing on $s$, yield that
\begin{eqnarray*}
\sum\limits_{r=1}^dB_r(\emph{\textbf{x}})\sum\limits_{s=1}^d
\frac{\partial T^{(r)}(\emph{\textbf{x}})}{\partial x_s}v_s(\emph{\textbf{x}})=-v_1(\emph{\textbf{x}}).
\end{eqnarray*}
Let $\emph{\textbf{x}}\uparrow \emph{\textbf{1}}$, we can see that the left hand side is nonnegative while the right hand side is $-v_1(\emph{\textbf{1}})<0$, which is a contradiction. Therefore, $\lambda_{\mathcal{C}}=0$.
\par
{\rm{(b)}}\ We now remove the condition $\rho(\emph{\textbf{1}})=0$. Suppose that $\lambda_*=-\rho(\emph{\textbf{q}}_*)$. Obviously, $q_{*r}>0\ (r=1,\cdots,d)$. We can define $\tilde{Q}=(\tilde{q}_{\emph{\textbf{ij}}}:\emph{\textbf{i}},\emph{\textbf{j}}\in \mathcal{C})$ as follows:
\begin{eqnarray*}
\tilde{q}_{\emph{\textbf{i}}\emph{\textbf{j}}}
   =\begin{cases}
\sum\limits_{r=1}^d\tilde{b}^{(r)}_{\emph{\textbf{j}}
-\emph{\textbf{i}}+\emph{\textbf{e}}_r}\chi_{_{\{i_r>0\}}},\ & \emph{\textbf{i}}, \emph{\textbf{j}}\in \mathcal{C}, \emph{\textbf{j}}-\emph{\textbf{i}}+\emph{\textbf{e}}_r\in \mathbf{Z}_+^d,\\
0,\ & \text{otherwise},
\end{cases}
\end{eqnarray*}
where $\tilde{b}^{(r)}_{\emph{\textbf{j}}}=b^{(r)}_{\emph{\textbf{j}}}
\emph{\textbf{q}}_*^{\emph{\textbf{j}}-\emph{\textbf{e}}_r}
+\delta_{\emph{\textbf{j}},\emph{\textbf{e}}_r}\lambda_*\ (\emph{\textbf{j}}\in \mathbf{Z}_+^d, r=1,\cdots,d)$ with $\{b^{(r)}_{\emph{\textbf{j}}}\}$ being given in (\ref{eq2.1}).
\par
Let $(\tilde{p}_{\emph{\textbf{ij}}}(t):\emph{\textbf{i}},\emph{\textbf{j}}\in \mathbf{C})$ be the minimal $\tilde{Q}$-function and $\tilde{B}_r(\emph{\textbf{x}})$ be the generating function of $\{\tilde{b}^{(r)}_{\emph{\textbf{j}}}:\emph{\textbf{j}}\in \mathbf{Z}_+^d\} \ (r=1,\cdots,d)$. Then
\begin{eqnarray*}
\tilde{B}_r(\emph{\textbf{x}})&=&\sum\limits_{\emph{\textbf{j}}\in \mathbf{Z}_+^d}\tilde{b}^{(r)}_{\emph{\textbf{j}}}\emph{\textbf{x}}
^{\emph{\textbf{j}}}\\
&=&\sum\limits_{\emph{\textbf{j}}\in \mathbf{Z}_+^d}(b^{(r)}_{\emph{\textbf{j}}}
\emph{\textbf{q}}_*^{\emph{\textbf{j}}-\emph{\textbf{e}}_r}
+\delta_{\emph{\textbf{j}},\emph{\textbf{e}}_r}\lambda_*)
\emph{\textbf{x}}
^{\emph{\textbf{j}}}\\
&=&q_{*r}^{-1}[B_r(\emph{\textbf{q}}_*\otimes\emph{\textbf{x}})
+\lambda_*q_{*r}x_r],\ \ \forall \emph{\textbf{x}}\in \prod\limits_{r=1}^d[0,\theta_r],
\end{eqnarray*}
where $\emph{\textbf{q}}_*\otimes\emph{\textbf{x}}:=(q_{*1}x_1,\cdots,q_{*d}x_d)$.
\par
Since $(\emph{\textbf{q}}_*^{\emph{\textbf{j}}}:\emph{\textbf{j}}\in \mathcal{C})$ is a $\lambda_*$-subinvariant vector for $Q$ on $\mathcal{C}$ and note that for any $\emph{\textbf{i}},\emph{\textbf{j}}\in \mathcal{C}$,
\begin{eqnarray*}
\tilde{q}_{\emph{\textbf{i}}\emph{\textbf{j}}}&=&
\sum\limits_{r=1}^d(b^{(r)}_{\emph{\textbf{j}}-\emph{\textbf{i}}
+\emph{\textbf{e}}_r}
\emph{\textbf{q}}_*^{\emph{\textbf{j}}-\emph{\textbf{i}}}
+\delta_{\emph{\textbf{i}},\emph{\textbf{j}}}\lambda_*)\chi_{_{\{i_r>0\}}}\\
&=&
\sum\limits_{r=1}^d(b^{(r)}_{\emph{\textbf{j}}-\emph{\textbf{i}}
+\emph{\textbf{e}}_r}
+\delta_{\emph{\textbf{i}},\emph{\textbf{j}}}\lambda_*)\chi_{_{\{i_r>0\}}}\emph{\textbf{q}}
_*^{\emph{\textbf{j}}-\emph{\textbf{i}}}\\
&=&(q_{\emph{\textbf{i}},\emph{\textbf{j}}}+\delta_{\emph{\textbf{i}},
\emph{\textbf{j}}}\lambda_*)\frac{\emph{\textbf{q}}
_*^{\emph{\textbf{j}}}}{\emph{\textbf{q}}_*^{\emph{\textbf{i}}}},
\end{eqnarray*}
by Lemma 5.4.2 of Anderson~\cite{And91}, we know that
\begin{eqnarray*}
\tilde{p}_{\emph{\textbf{i}}\emph{\textbf{j}}}(t)=\emph{\textbf{q}}
_*^{\emph{\textbf{j}}-\emph{\textbf{i}}}
p_{\emph{\textbf{i}}\emph{\textbf{j}}}(t)e^{\lambda_*t},\ \ \emph{\textbf{i}},\emph{\textbf{j}}\in \mathcal{C}.
\end{eqnarray*}
\par
Denote $\tilde{\lambda}_{\mathcal{C}}=\sup\{\lambda:
\int_0^{\infty}e^{\lambda t}\tilde{p}_{\emph{\textbf{i}}\emph{\textbf{j}}}(t)dt<\infty\}$, which is independent of $\emph{\textbf{i}},\emph{\textbf{j}}\in \mathcal{C}$ by Kingman~\cite{Kin63}. We claim that $\tilde{\lambda}_{\mathcal{C}}=0$. Indeed, suppose that $\tilde{\lambda}_{\mathcal{C}}>0$. It follows from Kolmogorov forward equations that
\begin{eqnarray*}
\sum_{\emph{\textbf{j}}\in \mathcal{C}}\tilde{p}'_{\emph{\textbf{ij}}}(t)\emph{\textbf{x}}^{\emph{\textbf{j}}}
=\sum_{r=1}^d\tilde{B}_r(\emph{\textbf{x}})\sum_{\emph{\textbf{j}}\in \mathcal{C}_r^+}
\tilde{p}_{\emph{\textbf{ij}}}(t)\emph{\textbf{x}}^{\emph{\textbf{j}}
-\emph{\textbf{e}}_r}-\sum\limits_{r=1}^d
\tilde{p}_{\emph{\textbf{i}}\emph{\textbf{e}}_r}(t)
\tilde{b}^{(r)}_{\emph{\textbf{0}}}.
\end{eqnarray*}
In particular,
\begin{eqnarray*}
-x_1
=\sum_{r=1}^d\tilde{B}_r(\emph{\textbf{x}})\sum_{\emph{\textbf{j}}\in \mathcal{C}_r^+}
\tilde{T}^{(r)}(\emph{\textbf{x}})-\sum\limits_{r=1}^d
p_{\emph{\textbf{e}}_1\emph{\textbf{e}}_r}(t)
\tilde{b}^{(r)}_{\emph{\textbf{0}}},
\end{eqnarray*}
where $\tilde{T}^{(r)}(\emph{\textbf{x}})=\sum\limits_{\emph{\textbf{j}}\in \mathcal{C}_r^+}
(\int_0^{\infty}\tilde{p}_{\emph{\textbf{e}}_1\emph{\textbf{j}}}(t)dt)\emph{\textbf{x}}^{\emph{\textbf{j}}
-\emph{\textbf{e}}_r}<\infty$. Hence,
\begin{eqnarray}\label{eq3.7}
-\delta_{s,1}
=\sum_{r=1}^d\tilde{B}_{rs}(\emph{\textbf{x}})\sum_{\emph{\textbf{j}}\in \mathcal{C}_r^+}
\tilde{T}^{(r)}(\emph{\textbf{x}})+\sum\limits_{r=1}^d
\tilde{B}_r(\emph{\textbf{x}})\frac{\partial \tilde{T}^{(r)}(\emph{\textbf{x}})}{\partial x_s},\ \ s=1,\cdots d.
\end{eqnarray}
\par
Define
\begin{eqnarray*}
\tilde{\lambda}_*=\sup\{\lambda\geq 0:\tilde{\emph{\textbf{B}}}(\emph{\textbf{x}})+\lambda \emph{\textbf{x}}\leq \emph{\textbf{0}}\ {\text{has\ a\ solution\ in}}\ [0,+\infty)^d\}.
\end{eqnarray*}
Then $\tilde{\lambda}_*=0$. Since $\emph{\textbf{q}}_*\in \prod\limits_{s=1}^d[0,\theta_s)$, the unique solution of $\tilde{\emph{\textbf{B}}}(\emph{\textbf{x}})=0$ is $\emph{\textbf{1}}<(\frac{\theta_1}{q_{*1}},\cdots,\frac{\theta_d}{q_{*d}})$. Therefore, by Theorem~\ref{th3.1}(ii), the maximal eigenvalue of $\tilde{\emph{\textbf{B}}}'(\emph{\textbf{1}})$ is $0$.
By Lemma~\ref{le2.3}, there exists a positive vector $\tilde{\emph{\textbf{v}}}=(\tilde{v}_1,\cdots,
\tilde{v}_d)$ such that
\begin{eqnarray*}
\sum_{r=1}^d\tilde{B}_{rs}(\emph{\textbf{1}})\cdot \tilde{v}_s=0,\ \ r=1,\cdots,d.
\end{eqnarray*}
Multiplying $\tilde{v}_l$ and then summing on $l$, yield that
 $-\tilde{v}_1=0$, which is a contradiction. Hence, $\tilde{\lambda}_{\mathcal{C}}=0$ and thus $\lambda_{\mathcal{C}}=\lambda_*$. The proof is complete.
\hfill $\Box$
\end{proof}
\par
The following theorem gives the exact value of $\lambda_{\mathcal{C}}$.
\par
\begin{theorem}\label{th3.3}
Suppose that $Q$ is the $q$-matrix given in $(\ref{eq1.2})$ and $P(t)=(p_{\textbf{ij}}(t):\textbf{i},\textbf{j}\in \mathbf{Z}_+^d)$ is the $Q$-function. Then $\lambda_{\mathcal{C}}=\lambda_*$.
\end{theorem}
\par
\begin{proof}
By Theorem~\ref{th3.2}, we only need to consider the case that $\lambda_*<-\rho(\emph{\textbf{q}}_*)$. Let $(\phi_{\emph{\textbf{ij}}}(\lambda):\emph{\textbf{i}},\emph{\textbf{j}}\in \mathbf{Z}_+^d)$ denote the Laplace transform of $(p_{\emph{\textbf{ij}}}(t):\emph{\textbf{i}},\emph{\textbf{j}}\in \mathbf{Z}_+^d)$. For any $n\geq 2$, define
\begin{eqnarray*}
q^{(n)}_{\emph{\textbf{i}}\emph{\textbf{j}}}=
\begin{cases}
\sum\limits_{r=1}^db^{(r,n)}_{\emph{\textbf{j}}-\emph{\textbf{i}}
+\emph{\textbf{e}}_r}\chi_{_{\{i_r>0,\emph{\textbf{j}}-\emph{\textbf{i}}
+\emph{\textbf{e}}_r\}}},\ & \text{if}\ \emph{\textbf{i}}\neq \emph{\textbf{0}},\ \emph{\textbf{j}}\in \mathbf{Z}_+^d,\\
0,\ & \text{otherwise},
\end{cases}
\end{eqnarray*}
where
$b^{(r,n)}_{\emph{\textbf{j}}}=
b^{(r)}_{\emph{\textbf{j}}}\chi_{_{\{j_r\leq n\}}}\ (r=1,\cdots,d, n\geq 2)$. It is obvious that $Q^{(n)}=(q^{(n)}_{\emph{\textbf{i}}\emph{\textbf{j}}}:\emph{\textbf{i}},
\emph{\textbf{j}}\in \mathbf{Z}_+^d)$ is a nonconservative generator matrix as in (\ref{eq1.2}) and
$\mathcal{C}$ is still communicating for each $Q^{(n)} (n\geq N_0)$.
Let $(_np_{\emph{\textbf{i}}\emph{\textbf{j}}}(t):\emph{\textbf{i}},
\emph{\textbf{j}}\in \mathbf{Z}_+^d)$ and
$(_n\phi_{\emph{\textbf{i}}\emph{\textbf{j}}}(\lambda):\emph{\textbf{i}},
\emph{\textbf{j}}\in{\bf Z_+})$ be the Feller minimal
$Q^{(n)}$-function and the Feller minimal $Q^{(n)}$-resolvent,
respectively. Define
\begin{eqnarray*}
   B_r^{(n)}(\emph{\textbf{x}})=\sum\limits_{\emph{\textbf{j}}\in \mathbf{Z}_+^d}b^{(r,n)}_{\emph{\textbf{j}}}
   \emph{\textbf{x}}^{\emph{\textbf{j}}},\ \ r=1,\cdots,d,\ n\geq N_0
\end{eqnarray*}
and
\begin{eqnarray*}
   \lambda_*^{(n)}=\sup\{\lambda\geq 0: \emph{\textbf{B}}^{(n)}(\emph{\textbf{x}})+\lambda \emph{\textbf{x}}=\emph{\textbf{0}} \ {\rm has \ a\ root\
   in}\
   [0,\infty)^{d}\},
\end{eqnarray*}
where $\emph{\textbf{B}}^{(n)}(\emph{\textbf{x}})=(B_1^{(n)}(\emph{\textbf{x}}),\cdots,
B_d^{(n)}(\emph{\textbf{x}}))$. Obviously, the
generating function $\emph{\textbf{B}}^{(n)}(\emph{\textbf{x}})$ is well-defined on $[0,\infty)^d$, i.e.,
the convergence radius of $\{B_r^{(n)}(\emph{\textbf{x}}):r=1,\cdots,d\}$ are all infinite. By Theorem~\ref{th3.1}, $\lambda_*^{(n)}= -\rho_0^{(n)}(\emph{\textbf{q}}_*^{(n)})$, where $\emph{\textbf{q}}_*^{(n)}$ is the unique solution of $\emph{\textbf{B}}^{(n)}(\emph{\textbf{x}})+\lambda_*^{(n)} \emph{\textbf{x}}=\emph{\textbf{0}}$ and $\rho_0^{(n)}(\emph{\textbf{q}}_*^{(n)})$ is the maximal eigenvalue of $(B_{rs}^{(n)}(\emph{\textbf{q}}_*^{(n)}):r,s=1\cdots,d)$. Therefore, by Theorem~\ref{th3.2}, the decay parameter of $(_np_{\emph{\textbf{i}}\emph{\textbf{j}}}(t):\emph{\textbf{i}},
\emph{\textbf{j}}\in \mathbf{Z}_+^d)$ for $\mathcal{C}$ is
$\lambda^{(n)}_{\mathcal{C}}=\lambda_*^{(n)}$.
\par
We now prove that
$\lambda_*^{(n)}\downarrow \lambda_*\ (n\uparrow \infty)$. Indeed, by the definition of $\emph{\textbf{B}}^{(n)}(\emph{\textbf{x}})$, we see that
\begin{eqnarray}
\label{eq3.8}
\emph{\textbf{B}}^{(n)}(\emph{\textbf{x}})\leq \emph{\textbf{B}}^{(n+1)}(\emph{\textbf{x}})\leq \emph{\textbf{B}}(\emph{\textbf{x}}),\ \
\ \forall\ \emph{\textbf{x}}\in \prod\limits_{s=1}^d[0,\theta_s],\ n\geq N_0.
\end{eqnarray}
It can be easily proved that
$\lambda_*^{(n)}\geq \lambda_*^{(n+1)}\geq \lambda_*\ (n\geq N_0)$. Therefore,
$\bar{\lambda}:=\lim\limits_{n\rightarrow \infty}\lambda_*^{(n)}\geq \lambda_*$. We further claim that
$\bar{\lambda}=\lambda_*$. Actually, if $\bar{\lambda}>\lambda_*$, then choose $\tilde{\lambda}\in (\lambda_*,\bar{\lambda})$. By Lemma~\ref{le3.1} and (\ref{eq3.8}), there exists an unique $\tilde{\emph{\textbf{x}}}^{(n)}\in \prod\limits_{s=1}^d[0,q_{*s}]$ such that
$\emph{\textbf{B}}^{(n)}(\tilde{\emph{\textbf{x}}}^{(n)})+\tilde{\lambda}
\tilde{\emph{\textbf{x}}}^{(n)}=\emph{\textbf{0}}$. By the definition of $\emph{\textbf{B}}^{(n)}(\emph{\textbf{x}})$ and the continuity of $\emph{\textbf{B}}(\emph{\textbf{x}})$, there exists $\tilde{\emph{\textbf{x}}}\in \prod\limits_{s=1}^d[0,q_{*s}]$ such that
$\emph{\textbf{B}}(\tilde{\emph{\textbf{x}}})+\tilde{\lambda}
\tilde{\emph{\textbf{x}}}=\emph{\textbf{0}}$, which contradicts the definition of $\lambda_*$. Hence, $\bar{\lambda}=\lambda_*$.
\par
On the other hand, it is well known that
$(\phi_{\emph{\textbf{ij}}}(\lambda):\emph{\textbf{i}},\emph{\textbf{j}}\in \mathbf{Z}_+^d)$ and $(_n\phi_{\emph{\textbf{ij}}}(\lambda):\emph{\textbf{i}},\emph{\textbf{j}}\in \mathbf{Z}_+^d)$ are the minimal nonnegative
solution of the Kolmogorov backward equations
\begin{eqnarray*}
   \phi_{\emph{\textbf{ij}}}(\lambda)=\frac{\delta_{\emph{\textbf{ij}}}}
   {\lambda-q_{\emph{\textbf{ii}}}}+\sum_{\emph{\textbf{l}}\neq
   \emph{\textbf{i}}}\frac{q_{\emph{\textbf{il}}}}{\lambda- q_{\emph{\textbf{ii}}}}\cdot \phi_{\emph{\textbf{lj}}}(\lambda),\ \ \
 \emph{\textbf{i}},\emph{\textbf{j}}\in \mathbf{Z}_+^d
\end{eqnarray*}
and
\begin{eqnarray*}
   _n\phi_{\emph{\textbf{ij}}}(\lambda)=\frac{\delta_{\emph{\textbf{ij}}
   }}{\lambda-q^{(n)}_{\emph{\textbf{ii}}}}+\sum_{\emph{\textbf{l}}\neq
   \emph{\textbf{i}}}\frac{q^{(n)}_{\emph{\textbf{il}}}}{\lambda-q^{(n)}
   _{\emph{\textbf{ii}}}}\cdot
   _n\phi_{\emph{\textbf{lj}}}(\lambda),\ \ \emph{\textbf{i}},\emph{\textbf{j}}\in \mathbf{Z}_+^d
\end{eqnarray*}
respectively. Furthermore, all of them can be obtained by the well-known
iteration scheme. Now note that $q^{(n)}_{\emph{\textbf{ii}}}=q_{\emph{\textbf{ii}}}\ (n\geq \tilde{N}_0)$ and
$q^{(n)}_{\emph{\textbf{il}}}\uparrow q_{\emph{\textbf{il}}} \ (n\uparrow \infty)$ for all $\emph{\textbf{i}}\neq
\emph{\textbf{l}}$. By considering their iteration schemes, we know that
$_n\phi_{\emph{\textbf{ij}}}(\lambda)\uparrow \phi_{\emph{\textbf{ij}}}(\lambda)$ as $n\uparrow
\infty$ and thus for their corresponding transition functions we
also have $_np_{\emph{\textbf{ij}}}(t)\uparrow p_{\emph{\textbf{ij}}}(t)$ as $n\uparrow \infty$. Therefore, $\lambda^{(n)}_*=\lambda^{(n)}_{\mathcal{C}}\geq \lambda_{\mathcal{C}}$ which implying $\lambda_*\geq \lambda_{\mathcal{C}}$. Hence, by Theorem~\ref{th3.2}, $\lambda_{\mathcal{C}}=\lambda_*$. The proof
is complete.
\hfill$\Box$
\end{proof}
\par
By Lemma~\ref{le3.1} and Theorems~\ref{th3.2}-\ref{th3.3}, we know that the decay parameter $\lambda_{\mathcal{C}}$ equals the largest $\lambda$ such that $\emph{\textbf{B}}(\emph{\textbf{x}})+\lambda \emph{\textbf{x}}=\emph{\textbf{0}}$ has a solution in $[0,\infty)^d$. We now consider how to obtain $\lambda_{\mathcal{C}}$. To this end, for any $k\in \{1,2,\cdots,d\}$, denote
\begin{eqnarray}\label{eq3.9}
\lambda^*_k:=\sup\{\frac{B_k(\emph{\textbf{x}})}{-x_k}:
\emph{\textbf{x}}\in [0,\infty)^d\}\ \ \text{conditioned\ on}\ x_kB_r(\emph{\textbf{x}})=x_rB_k(\emph{\textbf{x}}),\ r\neq k.
\end{eqnarray}
Since $\frac{B_k(\emph{\textbf{x}})}{-x_k}=(\lambda_k+\mu_k)-\Lambda_r(x_k)
-\frac{\Gamma_r(\emph{\textbf{x}})}{x_k}$, we know that the supremum in (\ref{eq3.9})can be achieved and hence we can write $\lambda^*_k:=\max\{\frac{B_k(\emph{\textbf{x}})}{-x_k}:
\emph{\textbf{x}}\in [0,\infty)^d\}$ in the following.
\par
\begin{theorem}\label{th3.4}
$\lambda_{\mathcal{C}}=\lambda_*=\min\{\lambda^*_1,\cdots,\lambda^*_d\}$.
\end{theorem}
\par
\begin{proof}
By Lemma~\ref{le3.1}, $\emph{\textbf{B}}(\emph{\textbf{x}})+\lambda_* \emph{\textbf{x}}=\emph{\textbf{0}}$ has a solution $\emph{\textbf{q}}_*=(q_{*1},\cdots,q_{*d})\in \prod\limits_{r=1}^d(0,\theta_r]$. It is easy to see that $q_{*k}B_r(\emph{\textbf{q}}_*)=q_{*r}B_k(\emph{\textbf{q}}_*)$ for any $r,k\in \{1,2,\cdots,d\}$. Therefore, $\lambda_*=\frac{B_k(\emph{\textbf{q}}_*)}{-q_{*k}}\leq \lambda^*_k$ for any $k\in \{1,2,\cdots,d\}$. Hence, $\lambda_{\mathcal{C}}=\lambda_*\leq \min\{\lambda^*_1,\cdots,\lambda^*_d\}$. Conversely, let $\lambda^*_{\tilde{k}}=\min\{\lambda^*_1,\cdots,\lambda^*_d\}$ and $\lambda^*_{\tilde{k}}$ is achieved at $\emph{\textbf{x}}_*=(x_{*1},\cdots,x_{*d})\in (0,\infty)^d$. Then
$\lambda^*_{\tilde{k}}=\frac{B_{\tilde{k}}(\emph{\textbf{x}}_*)}
{-x_{*\tilde{k}}}$ and $x_{\tilde{k}}B_r(\emph{\textbf{x}})=x_rB_{\tilde{k}}(\emph{\textbf{x}})$ for all $r\in \{1,\cdots,d\}$. This implies that $\emph{\textbf{x}}_*$ is a solution of $\emph{\textbf{B}}(\emph{\textbf{x}})+\lambda^*_{\tilde{k}}
\emph{\textbf{x}}=\emph{\textbf{0}}$. Therefore, by the definition of $\lambda_*$, we know that $\min\{\lambda^*_1,\cdots,\lambda^*_d\}=\lambda^*_{\tilde{k}}\leq \lambda_*=\lambda_{\mathcal{C}}$. The proof is complete. \hfill $\Box$
\end{proof}
\par
By Theorem~\ref{th3.4}, we see that the decay parameter $\lambda_{\mathcal{C}}$ can be obtained by finding the extreme value of $\frac{B_k(\emph{\textbf{x}})}{-x_k}$ conditioned on $x_kB_r(\emph{\textbf{x}})=x_rB_k(\emph{\textbf{x}})\ (r\in \{1,\cdots,d\})$ for $k\in \{1,\cdots,d\}$.
\par
\vspace{5mm}
 \setcounter{section}{4}
 \setcounter{equation}{0}
 \setcounter{theorem}{0}
 \setcounter{lemma}{0}
 \setcounter{corollary}{0}
\noindent {\large \bf 4. Transiency and invariant measure }
 \vspace{3mm}
 \par
 In this section, we further consider the transiency and invariant measure of $P(t)$. As in the previous section, let $\emph{\textbf{q}}(\lambda)=(q_1(\lambda),\cdots,q_d(\lambda))$ denote the unique solution of
 \begin{eqnarray*}
 \emph{\textbf{B}}(\emph{\textbf{x}})+\lambda \emph{\textbf{x}}=\emph{\textbf{0}}
 \end{eqnarray*}
 on $\prod\limits_{s=1}^d[0,q_{*s}]$ for $\lambda\in [0,\lambda_*]$.
 The following conclusion is our main result in this section.
  \par
\begin{theorem}
\label{th4.1}\ \ Suppose that $Q$ is the $q$-matrix given in $(\ref{eq1.2})$ and $P(t)=(p_{\textbf{ij}}(t):\textbf{i},\textbf{j}\in \mathbf{Z}_+^d)$ is the $Q$-function. Then for any $\lambda
\in [0,\lambda_{\mathcal{C}}]$ and $\textbf{i}\in \mathcal{C}$,
\begin{eqnarray}
\label{eq4.1}
    \int_0^{\infty}e^{\lambda t}p'_{\textbf{i0}}(t)dt\leq \textbf{q}^{\textbf{i}}(\lambda)
\end{eqnarray}
and
\begin{eqnarray}
\label{eq4.2}
  \int_0^{\infty}e^{\lambda t}p'_{\emph{\textbf{i0}}}(t)dt-\textbf{x}^{\textbf{i}}
  -\lambda \int_0^{\infty}e^{\lambda t}F_{\textbf{i}}(t,\textbf{x})dt=
\sum_{r=1}^dB_r(\textbf{x})\cdot \int_0^{\infty}e^{\lambda t}F^{(r)}_{\textbf{i}}(t,\textbf{x})dt.
\end{eqnarray}
Hence, $P(t)$
is $\lambda_{\mathcal{C}}$-transient.
\end{theorem}

\begin{proof}
For $\lambda\in[0,\lambda_C)$, we know that $\int_0^{\infty}e^{\lambda t}p_{\emph{\textbf{ij}}}(t)dt<\infty$ for all $\emph{\textbf{i}},\emph{\textbf{j}}\in \mathcal{C}$ and hence $\int_0^{\infty}e^{\lambda t}p'_{\emph{\textbf{i0}}}(t)dt<\infty$ for all $\emph{\textbf{i}}\in \mathcal{C}$. By Lemma~\ref{le2.1},

\begin{eqnarray*}
\int_0^{\infty}e^{\lambda t}p'_{\emph{\textbf{i0}}}(t)dt-\emph{\textbf{x}}^{\emph{\textbf{i}}}-\lambda \int_0^{\infty}e^{\lambda t}F_{\emph{\textbf{i}}}(t,\emph{\textbf{x}})dt=
\sum_{r=1}^dB_r(\emph{\textbf{x}})\cdot \int_0^{\infty}e^{\lambda t}F^{(r)}_{\emph{\textbf{i}}}(t,\emph{\textbf{x}})dt,\ \ \emph{\textbf{x}}\in \prod\limits_{s=1}^d[0,q_{*s}],
\end{eqnarray*}
which implies that (\ref{eq4.2}) holds for $\lambda \in [0,\lambda_{\mathcal{C}})$. Let $\emph{\textbf{x}}=\emph{\textbf{q}}(\lambda)$ in the above equality, we get (\ref{eq4.1}). Letting $\lambda\uparrow \lambda_{\mathcal{C}}$ in (\ref{eq4.1}) yields $\int_0^{\infty}e^{\lambda_{\mathcal{C}} t}p'_{\emph{\textbf{i0}}}(t)dt\leq \emph{\textbf{q}}_*^{\emph{\textbf{i}}}$ and hence $\int_0^{\infty}e^{\lambda_{\mathcal{C}} t}p_{\emph{\textbf{ij}}}(t)dt<\infty$ for all $\emph{\textbf{i}},\emph{\textbf{j}}\in \mathcal{C}$. Therefore, $P(t)$
is $\lambda_{\mathcal{C}}$-transient and (\ref{eq4.2}) holds for $\lambda \in [0,\lambda_{\mathcal{C}}]$.
 The proof is complete. \hfill $\Box$
\end{proof}
\par
Now, we turn our attention to the quasi-stationary distribution of
the stopped $M^X/M/1$ queue. We first consider the invariant
measures.
\par
\begin{theorem}
\label{th4.2}\ \ Suppose that $Q$ is the $q$-matrix given in $(\ref{eq1.2})$ and $P(t)=(p_{\textbf{ij}}(t):\textbf{i},\textbf{j}\in \mathbf{Z}_+^d)$ is the $Q$-function. Then for any $\lambda \in [0,\lambda_{\mathcal{C}}]$,
\par
{\rm (i)}\ there exists a $\lambda$-invariant measure $(m_{\textbf{i}}:\textbf{i}\in \mathcal{C})$
for $Q$ on $\mathcal{C}$. Moreover, the generating function of this $\lambda$-invariant measure
$M(\textbf{x})=\sum\limits_{\textbf{i}\in \mathcal{C}}m_{\textbf{i}}\textbf{x}^{\textbf{i}}$ satisfies
\begin{eqnarray}\label{eq4.3}
  \lambda M(\textbf{x})+\sum\limits_{r=1}^dB_r(\textbf{x})M^{(r)}(\textbf{x})
  =\sum\limits_{r=1}^dm_{\textbf{e}_r}b^{(r)}_{\textbf{0}},\ \ \ \textbf{x}\in\prod\limits_{s=1}^d[0,q_{s}),
\end{eqnarray}
\par
\ \ \ \ where $M^{(r)}(\textbf{x})=\sum\limits_{\textbf{i}\in \mathcal{C}^+_r}m_{\textbf{i}}\textbf{x}^{\textbf{i}-\textbf{e}_r}$ and $\{m_{\textbf{e}_r}:r=1,\cdots,d\}$ are positive constants.
\par
{\rm (ii)}\ This measure $(m_{\textbf{i}}:\textbf{i}\in \mathcal{C})$ is also a
$\lambda$-invariant for $P(t)$ on $\mathcal{C}$.
\par
{\rm (iii)}\ This $\lambda$-invariant measure is convergent {\rm
(}i.e., $\sum\limits_{\textbf{i}\in \mathcal{C}}m_{\textbf{i}}<\infty${\rm )} if and only if $\rho(\textbf{1})<0$, $\lambda_{\mathcal{C}}>0$ and $\lambda\in (0,\lambda_{\mathcal{C}}]$.
\end{theorem}
\par
\begin{proof}
Let $\lambda\in [0,\lambda_{\mathcal{C}}]$. It follows from Kolmogorov forward equations that
for any $\emph{\textbf{i}}\in \mathcal{C},\
\emph{\textbf{j}}\in \mathbf{Z}_+^d$,
\begin{eqnarray*}
p'_{\emph{\textbf{i}}
\emph{\textbf{j}}}(t)=\sum_{r=1}^d\sum_{\emph{\textbf{l}}\in \mathcal{C}_r^+}p_{\emph{\textbf{i}}
\emph{\textbf{l}}}(t)b^{(r)}_{\emph{\textbf{j}}
-\emph{\textbf{l}}+\emph{\textbf{e}}_r}.
\end{eqnarray*}
Therefore,
\begin{eqnarray}\label{eq4.4}
\int_{0}^{\infty}e^{\lambda t}p'_{\emph{\textbf{i}}
\emph{\textbf{0}}}(t)dt=\sum_{r=1}^d\int_0^{\infty}e^{\lambda t}p_{\emph{\textbf{i}}
\emph{\textbf{e}}_r}(t)dtb^{(r)}_{\emph{\textbf{0}}}
\end{eqnarray}
and for $\emph{\textbf{j}}\in \mathcal{C}$,
\begin{eqnarray}\label{eq4.5}
 \lambda \int_0^{\infty}e^{\lambda t}p_{\emph{\textbf{ij}}}(t)dt+\sum_{r=1}^d\sum_{\emph{\textbf{l}}\in \mathcal{C}_r^+}(\int_0^{\infty}e^{\lambda t}p_{\emph{\textbf{i}}
\emph{\textbf{l}}}(t)dt)b^{(r)}_{\emph{\textbf{j}}
-\emph{\textbf{l}}+\emph{\textbf{e}}_r}=-\delta_{\emph{\textbf{i}}
\emph{\textbf{j}}}.
\end{eqnarray}
Denote $m^{(\emph{\textbf{i}})}_{\emph{\textbf{j}}}=(\int_0^{\infty}e^{\lambda t}p'_{\emph{\textbf{i}}\emph{\textbf{0}}}(t)dt)^{-1}\int_0^{\infty}e^{\lambda t}p_{\emph{\textbf{i}}\emph{\textbf{j}}}(t)dt$ and
$\Delta^{(\emph{\textbf{i}})}_{\emph{\textbf{j}}}=(\int_0^{\infty}e^{\lambda
t}p'_{\emph{\textbf{i}}\emph{\textbf{0}}}(t)dt)^{-1}\delta_{\emph{\textbf{i}}
\emph{\textbf{j}}}$. Then (\ref{eq4.4}) and
(\ref{eq4.5}) can be rewritten as
\begin{eqnarray}\label{eq4.6} \sum\limits_{r=1}^db^{(r)}_{\emph{\textbf{0}}}m^{(\emph{\textbf{i}})}
_{\emph{\textbf{e}}_r}=1
\end{eqnarray}
and for $\emph{\textbf{j}}\in \mathcal{C}$,
\begin{eqnarray}\label{eq4.7}
\lambda m^{(\emph{\textbf{i}})}_{\emph{\textbf{j}}}+\sum\limits_{r=1}^d\sum_{\emph{\textbf{l}}\in \mathcal{C}^+_r}m^{(\emph{\textbf{i}})}_{\emph{\textbf{l}}}b^{(r)}_{\emph{\textbf{j}}
-\emph{\textbf{l}}+\emph{\textbf{e}}_r}=-\delta_{\emph{\textbf{i}}
\emph{\textbf{j}}}.
\end{eqnarray}
Since $m^{(\emph{\textbf{i}})}_{\emph{\textbf{e}}_r}\geq 0$,
it can be easily seen from (\ref{eq4.6}) that there exists a
subsequence $\emph{\textbf{i}}'$ such that $m_{\emph{\textbf{e}}_r}:=\lim_{\emph{\textbf{i}}'\rightarrow \infty}m^{(\emph{\textbf{i}}')}_{\emph{\textbf{e}}_r}<\infty$ for $b^{(r)}_{\emph{\textbf{0}}}>0$. As for $b^{(r)}_{\emph{\textbf{0}}}=0$, we may take $m_{\emph{\textbf{e}}_r}=1$. Then
\begin{eqnarray}
\label{eq4.8} \sum\limits_{r=1}^db^{(r)}_{\emph{\textbf{0}}}m
_{\emph{\textbf{e}}_r}=1.
\end{eqnarray}
For any $s=1,\cdots,d$, consider (\ref{eq4.7}) with $\emph{\textbf{j}}=\emph{\textbf{e}}_s$ and note
that (\ref{eq4.7}) also holds for $\emph{\textbf{i}}=\emph{\textbf{i}}'$, we can see
that there exist $0\leq m_{\emph{\textbf{e}}_{r}+\emph{\textbf{e}}_s}<\infty\ (r=1,\cdots,d)$ satisfying
\begin{eqnarray*}
\lambda m_{\emph{\textbf{e}}_s}
+\sum\limits_{r\neq s}m_{\emph{\textbf{e}}_r}b^{(r)}_{\emph{\textbf{e}}_s}
+\sum\limits_{r=1}^dm_{\emph{\textbf{e}}_s+\emph{\textbf{e}}_r}b^{(r)}
_{\emph{\textbf{0}}}=0.
\end{eqnarray*}
By mathematical
induction principle, we can obtain $(m_{\emph{\textbf{j}}}:\emph{\textbf{j}}\in \mathcal{C})$ which
are nonnegative and finite such that
\begin{eqnarray}\label{eq4.9}
\lambda m_{\emph{\textbf{j}}}+\sum\limits_{r=1}^d\sum_{\emph{\textbf{l}}\in \mathcal{C}^+_r}m_{\emph{\textbf{l}}}b^{(r)}_{\emph{\textbf{j}}
-\emph{\textbf{l}}+\emph{\textbf{e}}_r}=0,\ \ \ \emph{\textbf{j}}\in \mathcal{C}.
\end{eqnarray}
Now we claim that all $m_{\emph{\textbf{j}}}\ (\emph{\textbf{j}}\in \mathcal{C})$ are positive.
Indeed, if $m_{\tilde{\emph{\textbf{j}}}}=0$ for some $\tilde{\emph{\textbf{j}}}\in \mathcal{C}$, then by the
communicating property of $\mathcal{C}$, we know that $m_{\emph{\textbf{j}}}=0$ for all $\emph{\textbf{j}}\in \mathcal{C}$, which contradicts (\ref{eq4.8}). Therefore, $(m_{\emph{\textbf{j}}}:\emph{\textbf{j}}\in \mathcal{C})$ is a
$\lambda$-invariant measure for $Q$ on $\mathcal{C}$. Since $\int_0^{\infty}e^{\lambda_{\mathcal{C}} t}p'_{\emph{\textbf{i}}\emph{\textbf{0}}}(t)dt<\infty\ (\emph{\textbf{i}}\in \mathcal{C})$, by letting
$\lambda\uparrow \lambda_{\mathcal{C}}$ in (\ref{eq4.8})$-$(\ref{eq4.9}) and a similar argument as above, we get a $\lambda_{\mathcal{C}}$-invariant measure
for $Q$ on $\mathcal{C}$.
\par
Since $\lambda <\min\{-b^{(r)}_{\emph{\textbf{e}}_r}:r=1,\cdots,d\}$, multiplying
$\emph{\textbf{x}}^{\emph{\textbf{j}}}$ on both sides of (\ref{eq4.9}) and summing over
$\emph{\textbf{j}}\in \mathcal{C}$ yield that
\begin{eqnarray}
\label{eq4.10} \lambda M(\emph{\textbf{x}})+\sum\limits_{r=1}^dB_r(\emph{\textbf{x}})M^{(r)}(
\emph{\textbf{x}})
  =\sum\limits_{r=1}^dm_{\emph{\textbf{e}}_r}b^{(r)}_{\emph{\textbf{0}}}
\end{eqnarray}
for $\emph{\textbf{x}}$ in some neighbour-hood of $\emph{\textbf{0}}$. Note that
$\emph{\textbf{B}}(\emph{\textbf{x}})\neq \emph{\textbf{0}}$ for $\emph{\textbf{x}}\in \prod\limits_{s=1}^d[0,q_s)$, it is easily
seen that (\ref{eq4.10}) holds for $\emph{\textbf{x}}\in \prod\limits_{s=1}^d[0,q_s)$.
(i) is proved.
\par
In order to prove (ii), by Theorem~5.4.3 in
Anderson~\cite{And91}, we only need to prove that the equations
\begin{eqnarray}
\label{eq4.11}
\sum_{\emph{\textbf{i}}\in \mathcal{C}}y_{\emph{\textbf{i}}}q_{\emph{\textbf{ij}}}=-\mu y_{\emph{\textbf{j}}},\
\ 0\leq y_{\emph{\textbf{j}}}\leq m_{\emph{\textbf{j}}}, \ \emph{\textbf{j}}\in \mathcal{C}
\end{eqnarray}
have no nontrivial solution for some $\mu <\lambda$.
\par
For any $u<0$, since $\emph{\textbf{B}}(\emph{\textbf{q}})+u\emph{\textbf{q}}
=u\emph{\textbf{q}}<0$,
we know that
$\emph{\textbf{B}}(\emph{\textbf{x}})+u\emph{\textbf{x}}
=\emph{\textbf{0}}$ has exactly one solution $\emph{\textbf{q}}(u)\in \prod\limits_{s=1}^d[0,q_s)$ and $\emph{\textbf{q}}(u)\uparrow \emph{\textbf{q}}$ as $u\uparrow 0$.
\par
If $\lambda>0$, then let $\emph{\textbf{x}}=\emph{\textbf{q}}(u)$ in (\ref{eq4.10}) and let $u\uparrow 0$, we see that $M(\emph{\textbf{q}})=\lambda^{-1}
\sum\limits_{r=1}^dm_{\emph{\textbf{e}}_r}b^{(r)}_{\emph{\textbf{0}}}<\infty$. Suppose that
$(y_{\emph{\textbf{j}}}:\emph{\textbf{j}}\in \mathcal{C})$ is a nontrivial solution of (\ref{eq4.11}) with $\mu =0$. A similar argument as above yields
that all $y_{\emph{\textbf{j}}}$'s are positive. Then it follows from
(\ref{eq4.10}) that for any $\emph{\textbf{x}}\in \prod\limits_{s=1}^d[0,q_s]$,
\begin{eqnarray}
\label{eq4.12}
\sum\limits_{r=1}^dB_r(\emph{\textbf{x}})Y^{(r)}(
\emph{\textbf{x}})
  =\sum\limits_{r=1}^dy_{\emph{\textbf{e}}_r}b^{(r)}_{\emph{\textbf{0}}},
\end{eqnarray}
where $Y^{(r)}(\emph{\textbf{x}})=\sum\limits_{\emph{\textbf{j}}\in \mathcal{C}^+_r}y_{\emph{\textbf{j}}}\emph{\textbf{x}}
^{\emph{\textbf{j}}
-\emph{\textbf{e}}_r}$. This contradicts $y_{\emph{\textbf{j}}}\leq m_{\emph{\textbf{j}}}$.
\par
If $\lambda=0$. Suppose that $(y_{\emph{\textbf{j}}}:\emph{\textbf{j}}\in \mathcal{C})$ is a
nontrivial solution of (\ref{eq4.11}) with $\mu <0$. Similarly, we
have that all $y_{\emph{\textbf{j}}}$'s are positive and that for any $\emph{\textbf{x}}\in \prod\limits_{s=1}^d[0,q_s)$,
\begin{eqnarray}\label{eq4.13}
\mu Y(\emph{\textbf{x}})+\sum\limits_{r=1}^dB_r
(\emph{\textbf{x}})Y^{(r)}(
\emph{\textbf{x}})
  =\sum\limits_{r=1}^dy_{\emph{\textbf{e}}_r}b^{(r)}_{\emph{\textbf{0}}},
\end{eqnarray}
where $Y(\emph{\textbf{x}})=\sum\limits_{\emph{\textbf{j}}\in \mathcal{C}}y_{\emph{\textbf{j}}}\emph{\textbf{x}}^{\emph{\textbf{j}}}$, $Y^{(r)}(\emph{\textbf{x}})=\sum\limits_{\emph{\textbf{j}}\in \mathcal{C}^+_r}y_{\emph{\textbf{j}}}\emph{\textbf{x}}^{\emph{\textbf{j}}
-\emph{\textbf{e}}_r}$. Let $\emph{\textbf{x}}=\emph{\textbf{q}}(u)$ in the above equality and note that $y_{\emph{\textbf{j}}}\leq m_{\emph{\textbf{j}}}\ (\emph{\textbf{j}}\in \mathcal{C})$, by (\ref{eq4.10})
with $\lambda=0$ we know that
\begin{eqnarray*}
-\mu Y(\emph{\textbf{q}})\leq
\sum\limits_{r=1}^db^{(r)}_{\emph{\textbf{0}}}(m_{\emph{\textbf{e}}_r}
-y_{\emph{\textbf{e}}_r})<\infty,
\end{eqnarray*}
which contradicts (\ref{eq4.13}). (ii) is proved.
\par
Finally, suppose that $\rho(\emph{\textbf{1}})<0$, $\lambda_{\mathcal{C}}>0$ and $\lambda\in (0,\lambda_{\mathcal{C}}]$. Then $\emph{\textbf{q}}=\emph{\textbf{1}}$. Let $\emph{\textbf{x}}=\emph{\textbf{q}}(u)$ in (\ref{eq4.3}), we get
\begin{eqnarray*}
  \lambda M(\emph{\textbf{q}}(u))-u\sum\limits_{r=1}^dq_r(u)M^{(r)}
  (\emph{\textbf{q}}(u))
  =\sum\limits_{r=1}^db^{(r)}_{\emph{\textbf{0}}}m_{\emph{\textbf{e}}_r}.
\end{eqnarray*}
Let $u\uparrow 0$, we see that
$ M(\emph{\textbf{1}})\leq \lambda^{-1} \sum\limits_{r=1}^db^{(r)}_{\emph{\textbf{0}}}m_{\emph{\textbf{e}}_r}<\infty$.
\par
Conversely, suppose that $M(\emph{\textbf{1}})<\infty$. If $\lambda=0$, then by (\ref{eq4.3}) with $\emph{\textbf{x}}=\emph{\textbf{q}}(u)$, we have
\begin{eqnarray*}
\sum\limits_{r=1}^dM^{(r)}
  (\emph{\textbf{1}})=\lim\limits_{u\uparrow 0} \sum\limits_{r=1}^dq_r(u)M^{(r)}
  (\emph{\textbf{q}}(u))
  =-\lim\limits_{u\uparrow 0}u^{-1}\sum\limits_{r=1}^db^{(r)}_{\emph{\textbf{0}}}
  m_{\emph{\textbf{e}}_r}=\infty.
\end{eqnarray*}
Thus, $\lambda>0$ and hence $\lambda_{\mathcal{C}}>0$. By Lemma~\ref{le3.1}, for any $\lambda\in (0,\lambda_{\mathcal{C}}]$, $\emph{\textbf{B}}(\emph{\textbf{x}})+\lambda \emph{\textbf{x}}=\emph{\textbf{0}}$ has an unique solution $\emph{\textbf{q}}(\lambda)\in \prod\limits_{s=1}^d[0,q_{*s}]$. If $\rho(\emph{\textbf{1}})=0$, then we claim that $\emph{\textbf{q}}(\lambda)\in \prod\limits_{s=1}^d(1,q_{*s}]$. Indeed, if $H=\{r:q_r(\lambda)\leq 1\}\neq \emptyset$, then
\begin{eqnarray*}
B_r(\tilde{\emph{\textbf{u}}})\leq B_r(\emph{\textbf{q}}(\lambda))=-\lambda q_r(\lambda)<0,\ \ r\in H
\end{eqnarray*}
and
\begin{eqnarray*}
B_r(\tilde{\emph{\textbf{u}}})\leq B_r(\emph{\textbf{1}})=0,\ \ r\in \{1,\cdots,d\}\setminus H,
\end{eqnarray*}
where $\tilde{\emph{\textbf{u}}}=(\tilde{u}_k:k=1,\cdots,d)$ with $\tilde{u}_r=q_r(\lambda)\ (r\in H)$ and $\tilde{u}_r=1\ (r\in \{1,\cdots,d\}\setminus H)$. Therefore, there exists $\tilde{\emph{\textbf{x}}}\in \prod\limits_{s=1}^d[0,\tilde{u}_s]\subset [0,1]^d$ such that $\emph{\textbf{B}}(\tilde{\emph{\textbf{x}}})=0$. Since $\rho(\emph{\textbf{1}})=0$, we have $\tilde{\emph{\textbf{x}}}=\tilde{\emph{\textbf{u}}}=\emph{\textbf{1}}$ which contradicts $B_r(\tilde{\emph{\textbf{u}}})\leq -\lambda q_r(\lambda)<0 (r\in H)$. Therefore, $\emph{\textbf{q}}(\lambda)\in \prod\limits_{s=1}^d(1,q_{*s}]$. However, by Theorem~\ref{th3.1} and Lemma~\ref{le2.3}, $-\lambda\geq \rho(\emph{\textbf{q}}(\lambda))>\rho(\emph{\textbf{1}})=0$. This is a contradiction. If $\rho(\emph{\textbf{1}})>0$, then $\emph{\textbf{q}}\in [0,1)^d$. Since $M(\emph{\textbf{1}})<\infty$, we know that (\ref{eq4.3}) holds for $\emph{\textbf{x}}\in [0,1]^d$. Let $\emph{\textbf{x}}=\emph{\textbf{q}}$ and $\emph{\textbf{1}}$ in (\ref{eq4.3}), we get $\lambda M(\emph{\textbf{q}})=\sum\limits_{r=1}^dm_{\emph{\textbf{e}}_r}b^{(r)}
_{\emph{\textbf{0}}}=\lambda M(\emph{\textbf{1}})$, which is a contradiction. Therefore, $\rho(\emph{\textbf{1}})<0$. The proof is complete.\hfill $\Box$
\end{proof}

\vspace{5mm}
\setcounter{section}{5}
\setcounter{equation}{0}
\setcounter{theorem}{0}
\setcounter{lemma}{0}
\setcounter{corollary}{0}
\setcounter{remark}{0}
\noindent {\large \bf 5. An example}
\vspace{3mm}
\par
In this section, we will illustrate the practical implications of the above concepts (i.e., the decay parameter, invariant measure, and quasi-stationary distribution) by constructing a concrete example in auto quick repair service network.
\par
\begin{example}\label{exm5.1}
Consider an auto quick repair service network, consisting of two stations: downtown quick repair station (Station $A$), and suburban quick repair station (Station $B$). The customer flow rules of this network have the following characteristics:
\par
(1)\ {\bf Rules for external customer entry} ({\bf i.e., external load access rules})
\par
External customers of each station can be viewed as ``new vehicle repair needs" (such as citizens driving to the station for repairs or enterprise vehicles being sent for repair for the first time). They can only enter when the station has ``vehicles under repair" (i.e., there are customers). The entry rules are as follows.
\par
$\bullet$\
If Station $A$ currently has vehicles under repair (non-empty state), new external customers (e.g., private cars of downtown residents) are allowed to enter this station for repairs, the arrival rates are $\lambda_{11}=4.4$. The total arrival rate of Station $A$ is $\lambda_1=\lambda_{11}=4.4$.
\par
$\bullet$\
If Station $B$ currently has vehicles under repair (non-empty state), new external customers (e.g., private cars of downtown residents) are allowed to enter this station for repairs, the arrival rates are $\lambda_{21}=5.625$. The total arrival rate of Station $B$ is $\lambda_2=\lambda_{21}=5.625$.
\par
$\bullet$\ If a station has no vehicles under repair at present (empty state), the external entrance will be temporarily closed and no new external repair needs will be accepted (to avoid insufficient preparation of tools/personnel due to sudden repair acceptance, which would affect repair quality).
\par
(2)\ {\bf Rules of repair} ({\bf i.e., repair rate})
\par
The repair rates of Stations A and B are $\mu_1=2$ and $\mu_2=3.125$, respectively.
\par
(3)\ {\bf Rules for internal customer transfer} ({\bf i.e., load transfer})
\par
Internal customers of the system are ``vehicles under repair", which can be freely transferred between stations due to changes in repair needs, without any restriction. The transfer rates are as follows.
\par
$\bullet$\ When a vehicle at Station $A$ finishes its repair, it leaves the system with probability $\gamma_{10}=0.5$. Moreover, it transfers to Station $B$ with probability $\gamma_{12}=0.5$.
\par
$\bullet$\ When a vehicle at Station $B$ finishes its repair, it leaves the system with probability $\gamma_{20}=0.68$. Moreover, it transfers to Station $A$ with probability $\gamma_{21}=0.32$.
\par
It is easy to see that $B_1(x,y)=\frac{22}{5}x^2-\frac{32}{5}x+y+1$ and $B_2(x,y)=\frac{45}{8}y^2-\frac{35}{4}y+x+\frac{17}{8}$.
\par
By Lemma~\ref{le3.1} and Theorem \ref{th3.3}, the decay parameter $\lambda_{\mathcal{C}}$ is the largest $\lambda$ such that
\begin{eqnarray*}
\begin{cases}
\frac{22}{5}x^2-\frac{32}{5}x+y+1=-\lambda x\\
\frac{45}{8}y^2-\frac{35}{4}y+x+\frac{17}{8}=-\lambda y
\end{cases}
\end{eqnarray*}
has a nonnegative solution.
\par
By the expressions of $B_1(x,y)$ and $B_2(x,y)$, $B_1(x,y)=-\lambda x$ and $B_2(x,y)=-\lambda y$ can be rewritten as $\mathcal{C}_1:\ y=-\frac{22}{5}(x-\frac{32-5\lambda}{44})^2+
\frac{(32-5\lambda)^2}{440}-1$ and $\mathcal{C}_2:\ x=-\frac{45}{8}(y-\frac{35-4\lambda}{45})^2+
\frac{(35-4\lambda)^2}{360}-\frac{17}{8}$ respectively. We see that $\mathcal{C}_1$ is a downward-opening parabola with vertex $(\frac{32-5\lambda}{44},\frac{(32-5\lambda)^2}{440}-1)$ and $\mathcal{C}_2$ is a left-opening parabola with vertex $(\frac{(35-4\lambda)^2}{360}-\frac{17}{8},\frac{35-4\lambda}{45}
)$, where $\lambda\in [0,\frac{32}{5})$ is the index parameter. As $\lambda$ increases in $[0,\frac{32}{5})$, the two parabolas maintain their shape, and their vertices move in the left-downward direction. Let $f(\lambda)=\frac{32-5\lambda}{44}
-\frac{(35-4\lambda)^2}{360}+\frac{17}{8}$ and $g(\lambda)=\frac{(32-5\lambda)^2}{440}-1-\frac{35-4\lambda}{45}
$. Then $f(0)=-\frac{109}{198}, f(\frac{32}{5})=\frac{4229}{2250}>0$ and $g(0)=\frac{272}{495}>0,  g(\frac{32}{5})=-\frac{272}{225}<0$. Furthermore, $f'(\lambda)=\frac{263}{396}-\frac{4\lambda}{45}$ and $g'(\lambda)=\frac{5\lambda-32}{44}+\frac{4}{45}$.
Since the two parabolas intersect when $\lambda=0$ and are disjoint when $\lambda=\frac{32}{5}$, we know that there exists exactly one $\lambda_{max}\in [0,\frac{32}{5})$ such that the two parabolas intersect when $\lambda<\lambda_{max}$ and are disjoint when $\lambda>\lambda_{max}$, while they are tangent when $\lambda=\lambda_{max}$. It can be proved that $\lambda_{max}=1$ and the tangent point is $(0.5,0.6)$. Therefore, $\lambda_{\mathcal{C}}=1$.
\par
For the $\lambda_{\mathcal{C}}$-invariant measure $(m_{ij}:(i,j)\in \mathcal{C})$, by Theorem~\ref{th4.2}, the generating function $M(x,y):=\sum\limits_{i=1}^{\infty}\sum\limits_{j=0}^
{\infty}m_{ij}x^iy^j+\sum\limits_{j=1}^{\infty}m_{0j}y^j$ satisfies
\begin{eqnarray*}
&&M(x,y)+(\frac{22}{5}x^2-\frac{32}{5}x+y+1)
\frac{M(x,y)-M(0,y)}{x}\\
&&+(\frac{45}{8}y^2-\frac{35}{4}y+x+\frac{17}{8})\frac{M(x,y)-M(x,0)}{y}
=m_{10}+\frac{17}{8}m_{01}
\end{eqnarray*}
\begin{eqnarray*}
f(x,y)\cdot M(x,y)-yB_1(x,y)\cdot
M(0,y)-xB_2(x,y)\cdot
M(x,0)
=(m_{10}+\frac{17}{8}m_{01})xy,
\end{eqnarray*}
where $f(x,y):=\frac{22}{5}x^2y+\frac{45}{8}xy^2+x^2+y^2-
\frac{283}{20}xy+\frac{17}{8}x+y$ and $m_{10}, m_{01}$ are positive constants.
\end{example}
\par
\begin{remark}\label{re5.1}
First, the decay parameter measures how fast an auto quick repair service network transitions from operational state to shutdown state. It also helps analyze the decay of system workload, aiding recovery strategy-making. Second, the invariant measure reflects a system's long-term stable distribution under specific conditions, helping predict long-term load distribution across substations and aiding auto quick repair service network planning and expansion.
\end{remark}

\section*{Acknowledgement}
This work is supported by the National Key Research and Development Program of China (2022YFA1004600) and the National Natural Sciences Foundations of China (No. 11931018).
\par
\par

\end{document}